\documentclass[twoside,a4paper,11pt]{article}
\usepackage{amsmath}
\usepackage{amsthm}
\usepackage{amssymb}
\usepackage{amscd}
\usepackage{graphicx}
\usepackage{epsfig}
\usepackage{bm}

\usepackage{latexsym}
\usepackage{amsfonts}
\usepackage{color}

\input xy
\xyoption{all}

-\headheight\advance\topmargin by -\headsep

\font\black=cmbx10 \font\sblack=cmbx7 \font\ssblack=cmbx5 \font\blackital=cmmib10  \skewchar\blackital='177
\font\sblackital=cmmib7 \skewchar\sblackital='177 \font\ssblackital=cmmib5 \skewchar\ssblackital='177
\font\sanss=cmss11 \font\ssanss=cmss8 
\font\sssanss=cmss8 scaled 600 \font\blackboard=msbm10 \font\sblackboard=msbm7 \font\ssblackboard=msbm5
\font\caligr=eusm10 \font\scaligr=eusm7 \font\sscaligr=eusm5  \font\fraktur=eufm10
\font\sfraktur=eufm7 \font\ssfraktur=eufm5

\font\bsymb=cmsy10 scaled\magstep2
\def\all#1{\setbox0=\hbox{\lower1.5pt\hbox{\bsymb
			\char"38}}\setbox1=\hbox{$_{#1}$} \box0\lower2pt\box1\;}
\def\exi#1{\setbox0=\hbox{\lower1.5pt\hbox{\bsymb \char"39}}
	\setbox1=\hbox{$_{#1}$} \box0\lower2pt\box1\;}

\def\tx#1{{\fam0\relax#1}}

\newfam\bifam
\textfont\bifam=\blackital \scriptfont\bifam=\sblackital \scriptscriptfont\bifam=\ssblackital

\newfam\blfam
\textfont\blfam=\black \scriptfont\blfam=\sblack \scriptscriptfont\blfam=\ssblack

\newfam\bbfam
\textfont\bbfam=\blackboard \scriptfont\bbfam=\sblackboard \scriptscriptfont\bbfam=\ssblackboard

\newfam\ssfam
\textfont\ssfam=\sanss \scriptfont\ssfam=\ssanss \scriptscriptfont\ssfam=\sssanss
\def\sss#1{{\fam\ssfam\relax#1}}

\newfam\clfam
\textfont\clfam=\caligr \scriptfont\clfam=\scaligr \scriptscriptfont\clfam=\sscaligr

\newfam\frfam
\textfont\frfam=\fraktur \scriptfont\frfam=\sfraktur \scriptscriptfont\frfam=\ssfraktur

\def\hpb#1{\setbox0=\hbox{${#1}$}
	\copy0 \kern-\wd0 \kern.2pt \box0}
\def\vpb#1{\setbox0=\hbox{${#1}$}
	\copy0 \kern-\wd0 \raise.08pt \box0}

\def\pmb#1{\setbox0\hbox{${#1}$} \copy0 \kern-\wd0 \kern.2pt \box0}
\def\pmbb#1{\setbox0\hbox{${#1}$} \copy0 \kern-\wd0
	\kern.2pt \copy0 \kern-\wd0 \kern.2pt \box0}
\def\pmbbb#1{\setbox0\hbox{${#1}$} \copy0 \kern-\wd0
	\kern.2pt \copy0 \kern-\wd0 \kern.2pt
	\copy0 \kern-\wd0 \kern.2pt \box0}
\def\pmxb#1{\setbox0\hbox{${#1}$} \copy0 \kern-\wd0
	\kern.2pt \copy0 \kern-\wd0 \kern.2pt
	\copy0 \kern-\wd0 \kern.2pt \copy0 \kern-\wd0 \kern.2pt \box0}
\def\pmxbb#1{\setbox0\hbox{${#1}$} \copy0 \kern-\wd0 \kern.2pt
	\copy0 \kern-\wd0 \kern.2pt
	\copy0 \kern-\wd0 \kern.2pt \copy0 \kern-\wd0 \kern.2pt
	\copy0 \kern-\wd0 \kern.2pt \box0}

\mathchardef\za="710B  
\mathchardef\zb="710C  
\mathchardef\zg="710D  
\mathchardef\zd="710E  
\mathchardef\zve="710F 
\mathchardef\zz="7110  
\mathchardef\zh="7111  
\mathchardef\zvy="7112 
\mathchardef\zi="7113  
\mathchardef\zk="7114  
\mathchardef\zl="7115  
\mathchardef\zm="7116  
\mathchardef\zn="7117  
\mathchardef\zx="7118  
\mathchardef\zp="7119  
\mathchardef\zr="711A  
\mathchardef\zs="711B  
\mathchardef\zt="711C  
\mathchardef\zu="711D  
\mathchardef\zvf="711E 
\mathchardef\zq="711F  
\mathchardef\zc="7120  
\mathchardef\zw="7121  
\mathchardef\ze="7122  
\mathchardef\zy="7123  
\mathchardef\zf="7124  
\mathchardef\zvr="7125 
\mathchardef\zvs="7126 
\mathchardef\zf="7127  
\mathchardef\zG="7000  
\mathchardef\zD="7001  
\mathchardef\zY="7002  
\mathchardef\zL="7003  
\mathchardef\zX="7004  
\mathchardef\zP="7005  
\mathchardef\zS="7006  
\mathchardef\zU="7007  
\mathchardef\zF="7008  
\mathchardef\zW="700A  

\newcommand{\be}{\begin{equation}}
\newcommand{\ee}{\end{equation}}
\newcommand{\ra}{\rightarrow}

\newcommand{\bea}{\begin{eqnarray}}
\newcommand{\eea}{\end{eqnarray}}
\newcommand{\beas}{\begin{eqnarray*}}
	\newcommand{\eeas}{\end{eqnarray*}}

\def\*{{\textstyle *}}
\newcommand{\R}{{\mathbb R}}
\newcommand{\T}{{\mathsf T}}

\newcommand{\we}{\wedge}

\newcommand{\pa}{\partial}
\newcommand{\ti}{\times}

\def\ran{\rangle}

\def\Sec{\operatorname{Sec}}

\def\la{\langle}
\def\ran{\rangle}

\def\sT{{\sss T}}

\def\xd{\tx{d}}
\def\xi{\tx{i}}

\def\rel{{-\!\!\!-\!\!\rhd}}
\newdir{|>}{%
	!/4.5pt/@{|}*:(1,-.2)@^{>}*:(1,+.2)@_{>}}

\def\Graph{\operatorname{graph}}

\def\pr{\operatorname{pr}}

\def\Id{\operatorname{Id}}

\def\graph{\operatorname{graph}}

\def\rel{{-\!\!\!-\!\!\rhd}}
\newdir{ (}{{}*!/-5pt/@^{(}}

\newcommand{\bra}[1]{\ensuremath{\langle #1 |}}

\newcommand{\nm}[1]{\ensuremath{\Vert #1 \Vert}}

\newtheorem{theorem}{Theorem}[section]
\newtheorem{proposition}[theorem]{Proposition}
\newtheorem{lemma}[theorem]{Lemma}

\newtheorem{conjecture}[theorem]{Conjecture}

\theoremstyle{definition}
\newtheorem{example}[theorem]{Example}
\newtheorem{definition}[theorem]{Definition}
\newtheorem{remark}[theorem]{Remark}

\newtheorem{question}[theorem]{Question}

\newcommand{\Gr}{\textnormal{Gr}}

\def\cF{\mathcal{F}}

\newcommand{\lcf}{\lbrack\! \lbrack}
\newcommand{\rcf}{\rbrack\! \rbrack}
\newcommand{\lvec}[1]{\overleftarrow{#1}}
\newcommand{\rvec}[1]{\overrightarrow{#1}}

\newcommand{\Oc}{{\mathbb O}}

\begin{document}
	\title{\bf Nonassociative analogs  \\ of Lie groupoids \thanks{Research  founded by the  Polish National Science Centre grant
			under the contract number 2016/22/M/ST1/00542.
	}}
	\date{}
	\author{\\ Janusz Grabowski\\ Zohreh Ravanpak
		\\ \\
		{\it Institute of Mathematics}\\
		{\it Polish Academy of Sciences}
	}
	\maketitle
	
	\begin{abstract} We introduce nonassociative geometric objects generalising naturally Lie groupoids and called (smooth) \emph{quasiloopoids} and \emph{loopoids}.  We prove that the tangent bundles of smooth loopoids are canonically smooth loopoids again (it is nontrivial in the case of loopoids). We show also that this is not true if the cotangent bundles are concerned. After providing a few natural constructions, we show how the Lie-like functor associates with loopoids skew-algebroids and almost Lie algebroids and how discrete mechanics on Lie groupoids can be reformulated in the nonassociative case.
	\end{abstract}
	\date{}
	
	\section{Introduction}
	The r\^ole of the theory of Lie groups and Lie algebras in mathematics and theoretical physics cannot be overestimated.
	It is a fundamental tool for describing symmetries, starting from smooth geometry through the theory of smooth equations, and ending in field and quantum theories.
	
	The classical Lie theory has been extended to \emph{Lie groupoids} and \emph{Lie algebroids} (for the survey and literature see the monograph \cite{Ma} or \cite{Mei}) which turned out to be even more useful concepts if the geometry is concerned.
	Especially, symplectic groupoids as `non-commutative phase spaces', integrating Poisson structures and introduced by Weinstein, Karasev and Zakrzewski \cite{We,Zak}, give rise to numerous applications and fundamental geometric problems. This generalized Lie theory combined with symplectic geometry is one of the two main streams in geometry.
	
	The recent development has shown the importance of nonassociative algebraic
	structures, such as \emph{quasigroups} and \emph{loops}. For instance, it is possible to say that
	nonassociativity is the algebraic equivalent of the smooth geometric concept of curvature. The nonassociative generalizations of Lie groups are smooth loops. It is clear what a smooth loop is, on a manifold $G$ we have a smooth multiplication (no associativity is assumed) with unit $e\in G$ such that the left and right translations are diffeomorphisms. In \cite{GR} we developed discrete mechanics on smooth loops, but there is still a need to consider nonassociative objects closer to these of Lie groupoids; we will call them smooth \emph{(quasi)loopoids}. So, \emph{loopoids}, or their weaker version \emph{quasiloopoids}, will be related to Lie groupoids as smooth loops are related to Lie groups.
	
 We show that a Lie-like functor can be defined for quasiloopoids and loopoids and it will lead to new applications. The theory of loopoids has to be developed from scratch. The rough idea of a loopoid is clear, however the precise definition has to be an object of a deeper study. The definition should be based on important examples and determined by conceptual completeness. We shall use methods of differential geometry and algebra. Applications will be related to Lagrangian discrete mechanics (geometric integrators) and information geometry (statistical manifolds).
	
Note that here by a \emph{groupoid} we understand a \emph{Brandt groupoid}, i.e. a small category in which every morphism is an isomorphism, and not as an object which is called \emph{magma} in algebra.

 In the case of genuine groupoids, since the multiplication is only partially defined, the situation is more complicated. So the axioms of a loopoid must be essentially reformulated.
	
		A convenient way of thinking about groupoids is by using group axioms with the difference that all maps are replaced by relations, like it has been done by Zakrzewski \cite{Zak}. In particular, the unity is a relation $\ze:\{ e\}\rel\ G$\,, associating to a point $e$ a subset $M=\ze(e)\subset G$\,, the set of units. Shortly,
a smooth loopoid is a smooth manifold $G$ with surjective submersions $\za,\zb:G\to M$\,, on the set of units $M$\,, and a partially defined smooth multiplication $$m:G\ti G\supset G^{(2)}\ni (g,h)\mapsto gh\in G$$ where $$G^{(2)}=\{(g,h)\in G\ti G\,|\zb(g)=\za(h)\}\,,$$  such that the composition satisfies
$$\za(g)g=g\,,\ h\zb(h)=h\,,\ \za(gh)=\za(g)\,,\ \text{and}\ \zb(gh)=\zb(h)$$
(this reduces to the existence of an identity element for smooth loops). Note that no associativity or invertibility is assumed. Moreover, the canonical translations
$$l_g:\cF^\za(\zb(g))\to\cF^\za(\za(g))\,, x\mapsto gx\,,\quad\text{and}\quad r_h:\cF^\zb(\za(h))\to\cF^\zb(\zb(h))\,,
x\mapsto xh\,,$$
where $\cF^\za(x)=\za^{-1}(x)$ and $\cF^\zb(x)=\zb^{-1}(x)$ denote the corresponding $\za$ and $\zb$-fibers, are supposed to be diffeomorphisms. Again, for smooth loops, the latter conditions just mean that the left and the right translations are bijections. In this sense, a smooth loopoid over one point is a smooth loop.

We prove that for a smooth loopoid $\za,\zb:G\to M$ the tangent bundle $\sT\za,\sT\zb:\sT G\to\sT M$ is canonically a smooth loopoid with the multiplication $\sT m$\,. Contrary to the case of Lie groupoids this is a non-trivial theorem (we do not have inversion). Also the cotangent bundle to a Lie groupoid is canonically a Lie groupoid. We will show that this cannot be extend to the case smooth loopoids.
	
One of the aims of this paper is to extend the Lie functor to categories of nonassociative objects, like quasiloopoids and loopoids. Infinitesimal counterparts of Lie groupoids are Lie algebroids and the corresponding Lie theory is well established (cf. \cite{Ma}). As the infinitesimal version of associativity is the Jacobi identity, the infinitesimal counterpart of loopoids will not satisfy the latter.

Note that in the literature there are already natural various generalizations of Lie algebroids, e.g. \emph{skew algebroids}, \emph{almost Lie algebroids}, or \emph{Dirac algebroids} \cite{GG,GG1,GJ,GU1,GU2}, where no Jacobi identity is assumed. For instance, the skew algebroid formalism is very useful in describing the geometry of nonholonomic systems \cite{GLMM}.
	
As the Jacobi identity is an infinitesimal version of associativity, the `integrated version' of these generalizations must be understood as `nonassociative Lie groupoids`. The development of the corresponding theory is still a challenging task. This is a natural evolution of mathematical studies that not only we go deeper in our understanding of the known structures, but also we start to penetrate new lands of more general structures, hoping to extend the old results and discover new mathematical phenomena. Note that the nonassociativity is actually a form of noncommutativity: the left translations do not commute with the right translations.	

We want to stress that our motivation comes from discrete mechanics, where Lie groupoids have been recently used for a geometric formulation of the Lagrangian formalism \cite{FZ,IMMM,IMMP,MMM,MMS,MW,Sim,Stern,weinstein}. We expect that a natural application of smooth loopoids will be that in discrete Lagrangian mechanics. Another application can be found in information geometry. We have shown in \cite{GGKM} that a natural framework for information geometry is provided by Lie groupoids and Lie algebroids. We believe this can be, at least partially, generalized to a theory on smooth quasiloopoids and loopoids.

	A nonassociative algebra, e.g. the theory of algebraic loops, is already quite well developed.
		On the other way, many authors limit their investigation to loops that satisfy various other structural conditions. Common examples of such notions are the \emph{left-} and \emph{right-Bol loops}, the \emph{Moufang loops} (which are both left-Bol loops and right-Bol loops simultaneously), the \emph{generalized Bol loop}, etc. Especially the theory of analytic Moufang loops and Mal'cev algebras \cite{mal} is quite close to the Lie theory with respect to a generalized Lie functor \cite{kuz,nagy}. For general analytical loops the tangent algebra has a structure of an \emph{Akivis algebra} \cite{HS}.
	
	The problem is that the axioms for smooth Moufang loops are quite restrictive and these algebraic concepts do not fit very well to the needs of Lie groupoid geometry. One wants to have an object similar to a Lie groupoid, so with a product which is only partially defined and nonassociative.

	The question is, how far we can go with known constructions from the Lie theory in the nonassociative case. Investigating it will show what aspects are specific for the structures which are associative (or satisfying the Jacobi identity) and what we can get by assuming weaker axioms.

In the case of smooth loops, the theory resembles to some extent the theory of Lie groups and is pretty well-developed (see e.g. the book \cite{nagy1}). Of course, nonassociative  generalizations of Lie groupoids
provide much more problems. Some of them we discuss in this paper.
	
	Contrary to the case of smooth loops (e.g. \cite{FN,mal,kuz,Ne,NM, Sab2}), the literature in this subject, oriented on `nonassociative Lie groupoids' is not very extensive. Besides some aspects contained in Sabinin's monograph \cite{Sab}, we can indicate our short introductory note \cite{Gr} and the recent paper \cite{Alonso} (which appeared after the first version of our paper was prepared). In the latter paper, the notion of a \emph{quasigroupoid} is introduced. There are obvious common points with our concept of a loopoid (e.g. the source and target maps), however, the authors consider only purely algebraic cases (no differential geometry) with the emphasis on the cases when the sets of units are finite.

Note that the term \emph{loopoid} has already appeared in a paper by Kinyon \cite{Kin} in a similar context. The motivating example, however, built as an object `integrating' the Courant bracket on $\T M\oplus_M\T^*M$\,, uses the group of diffeomorphisms of the manifold $M$ as integrating the Lie algebra of vector fields on $M$, not the pair groupoid $M\ti M$ as `integrating' the Lie algebroid $\sT M$\,.

	\section{Quasigroups and loops}
	
	The remarkable development of smooth quasigroups and loops theory since the pioneering
	works of Mal'cev in 1955 (see \cite{mal}) was presented by Lev V. Sabinin in \cite{Sab}, where the large bibliography on the subject
	is given. We refer also to the books \cite{Bel,Pfl} and the survey articles \cite{Sab1,Smi} if terms and
	concepts from nonassociative algebra are concerned.
	
	Let us recall that a \emph{quasigroup} is an algebraic structure $<G,\cdot>$ with a binary operation (written usually as juxtaposition, $a\cdot b=ab$) such
	that
	$r_g: x \mapsto xg$ (the \emph{right translation}) and
	$l_g: x \mapsto gx$ (the \emph{left translation})
	are permutations of $G$, equivalently, in which the equations
	$ya = b$ and $ax = b$ are solvable uniquely for $x$ and $y$ respectively. A \emph{loop} is a quasigroup with a two-sided identity element,
	$e$, $e x=x e=x$. A loop $<G,\cdot , e>$ with identity $e$ is called an \emph{inverse
		loop} (or \emph{I.P. loop}) if to each element $a$ in G there corresponds an element $a^{-1}$ in $G$ such that
$$a^{-1}(a b) = (b a) a^{-1} =b\,,$$ for all $b\in G$\,.
	It can be then easily shown that in an inverse loop $<G, \cdot, {}^{-1}, e>$ we have,
	for all $a, b \in G$\,,
	$$aa^{-1} = a^{-1}a = e\,,\quad (a^{-1})^{-1} = a\,,\quad \text{and}\quad  (ab)^{-1} = b^{-1} a^{-1}\,.$$
	Indeed, the first two identities are trivial and
	$$[x^{-1}(xy)=y]\,\Rightarrow [x^{-1}=y(xy)^{-1}]\, \Rightarrow [y^{-1}x^{-1}=(xy)^{-1}]\,.$$
	Mimicking the Lie theory, one can define for any smooth loop, a `Lie functor' associating  with the loop a \emph{skew algebra}, i.e. a real vector space with a bilinear skew operation \cite{GR,Sab,Sab1}.
	
	\begin{example}\label{e1} The octonions $\mathbb{O}$ are the noncommutative non-associative algebra which is one of the four division algebras that exist over the real numbers. Every octonion can be expressed in terms of a natural basis $\{e_0,e_1,\cdot,\cdot,\cdot,e_7\}$ where $e_0=1$ represents the identity element and the imaginary octonion units $e_i$, $\{i=1,...,7\}$ satisfy the multiplication rule  $e_i e_j=-\zd_{i}^{j}+f_{ijk}e_k$\,, where $\zd_{i}^{j}$ is the Kronecker's delta and $f_{ijk}$'s are completely anti-symmetric structure constants which read as
			\[
			f_{123} =  f_{147} =  f_{165} =  f_{246 }=  f_{257} =  f_{354} =  f_{367} = 1\,.
			\]
			The multiplication is subject to the relations
			\[
			\forall i\ne 0\quad [e_i^2=-1]\,, \qquad
			e_ie_j=-e_je_i\,,\quad \mbox{for} \quad i\ne j\ne 0\,.
			\]
			and the following multiplication table.
		
		\begin{center}
			\begin{tabular}{| l || l | l | l | l | l | l | l | l | l | p{15mm} }
				
				\hline\hline
				\vspace{-1mm}

				{\scriptsize $e_ie_j$}&{\scriptsize $~e_0$}
				& {\scriptsize $~e_1$ }& {\scriptsize $~e_2$ }& {\scriptsize $~e_3$ }& {\scriptsize $~e_4$ }& {\scriptsize $~e_5$ }& {\scriptsize $~e_6$ }& {\scriptsize $~e_7$ }  \smallskip\\
				\hline\hline
				
				{\scriptsize $~e_0$ }&{\scriptsize $~e_0$}
				& {\scriptsize $~e_1$ }& {\scriptsize $~e_2$ }& {\scriptsize $~e_3$ }& {\scriptsize $~e_4$ }& {\scriptsize $~e_5$ }& {\scriptsize $~e_6$ }& {\scriptsize $~e_7$ } \\
				\hline
				{\scriptsize $~e_1$ }&{\scriptsize $~e_1$}
				& {\scriptsize $-e_0$ }& {\scriptsize $~e_3$ }& {\scriptsize $-e_2$ }& {\scriptsize $~e_5$ }& {\scriptsize $-e_4$ }& {\scriptsize $-e_7$ }& {\scriptsize $~e_6$ } \\
				\hline
				{\scriptsize $~e_2$}&{\scriptsize $~e_2$}
				& {\scriptsize $-e_3$ }& {\scriptsize $-e_0$ }& {\scriptsize $~e_1$ }& {\scriptsize $~e_6$ }& {\scriptsize $~e_7$ }& {\scriptsize $-e_4$ }& {\scriptsize $-e_5$ }  \\
				\hline
				{\scriptsize $~e_3$}&{\scriptsize $~e_3$}
				& {\scriptsize $~e_2$ }& {\scriptsize $-e_1$ }& {\scriptsize $-e_0$ }& {\scriptsize $~e_7$ }& {\scriptsize $-e_6$ }& {\scriptsize $~e_5$ }& {\scriptsize $-e_4$ }  \\
				\hline
				{\scriptsize $~e_4$ }&{\scriptsize $~e_4$}
				& {\scriptsize $-e_5$ }& {\scriptsize $-e_6$ }& {\scriptsize $-e_7$ }& {\scriptsize $-e_0$ }& {\scriptsize $~e_1$ }& {\scriptsize $~e_2$ }& {\scriptsize $~e_3$ }  \\
				\hline
				{\scriptsize$~e_5$ }&{\scriptsize $e_5$}
				& {\scriptsize $~e_4$ }& {\scriptsize $-e_7$ }& {\scriptsize $~e_6$ }& {\scriptsize $-e_1$ }& {\scriptsize $-e_0$ }& {\scriptsize $-e_3$ }& {\scriptsize $~e_2$ }  \\
				\hline
				{\scriptsize$~e_6$ }&{\scriptsize $~e_6$}
				& {\scriptsize $~e_7$ }& {\scriptsize $~e_4$ }& {\scriptsize $-e_5$ }& {\scriptsize $-e_2$ }& {\scriptsize $~e_3$ }& {\scriptsize $-e_0$ }& {\scriptsize $-e_1$ }  \\
				\hline
				{\scriptsize$~e_7$ }&{\scriptsize $~e_7$}
				& {\scriptsize $-e_6$ }& {\scriptsize $~e_5$ }& {\scriptsize $~e_4$ }& {\scriptsize $-e_3$ }& {\scriptsize $-e_2$ }& {\scriptsize $~e_1$ }& {\scriptsize $-e_0$ }  \\
				\hline
				
			\end{tabular}
		\end{center}
		
\noindent The associator $[g,h,k] =(gh)k - g(hk)$ of three octonions does not vanish in general but octonions satisfy a weak form of associativity known as alternativity, namely $[g,h,g]=0$.
		The octonions are a generalization of the complex numbers, with seven imaginary units, so octonionic conjugation is given by reversing the sign of the imaginary basis units. Conjugation is an involution of $\mathbb{O}$  satisfying $(gh)^* = h^* g^*$. The inner product on $\mathbb{O}$ is inherited from $\mathbb R^8$ and can be rewritten
		$$
		\left\langle g,h\right\rangle = \frac{(gh^* + hg^*)}{2} = \frac{(h^*g + g^*h)}{2}\in\R\,,
		$$
		and the norm of an octonion is just $\|g \| ^2 = gg^*$ which satisfies the defining property of a normed division algebra, namely $\|gh\| = \| g\| \| h \|$\,.
		The scalar product is invariant with respect to the multiplication: $\langle ag,ah\rangle=\langle g,h\rangle$ for $a\ne 0$\,.
		
\noindent Every nonzero octonion  $g\in \mathbb{O}$ has an inverse $g^{-1}=\frac{g^*}{\| g \| ^2}$, such that
	\be\label{inverse}
		gg^{-1}=g^{-1}g=1\,,	
		\ee
		which makes the set of invertible octonions into an inverse loop with respect to the octonion multiplication.
		We remark that the inverse in this loop is a genuine one, i.e. it satisfies the \emph{Inverse Property}:
	\[
		g(g^{-1}h)=g^{-1}(gh)=h\,, \quad \forall g,h\in \mathbb{O}\,,
		\]
		which is stronger than the standard property (\ref{inverse}) for non-associative algebra. 	
		Actually, the set $\Oc^\ti$ of invertible octonions is a smooth Moufang loop under octonion multiplication.

	\end{example}
	
	\begin{example}\label{ex1}	
		The set of all automorphisms of the algebra $\mathbb{O}$\,, that is the set of invertible linear transformations $A\in Aut (\mathbb{O})$\,, forms a Lie group called $G_2$ which is the smallest of the exceptional Lie groups. We will show that the semidirect product $\mathbb{O}^\ti \ltimes G_2 $ is an inverse loop under the multiplication
		\[
		(g,A)\bullet(h,B)=(g A(h), A\circ B)\,,
		\]
		with identity $(1,\Id)$ and inverse $(g,A)^{-1}=(A^{-1}(g^{-1}),A^{-1})$\,. What needs to be checked is the following inverse property,
		\[
		\begin{array}{rcl}
		(g,A)^{-1}\bullet \left( (g,A)\bullet (h,B)\right) &=&(A^{-1}(g^{-1}),A^{-1})\bullet (g\cdot A(h), A\circ B)\\[3pt]
		&=& (A^{-1}(g^{-1})\cdot A^{-1}(g\cdot A(h)),B)=(h,B)\,.
		\end{array}
		\]
		Here we use the fact that $A^{-1}(g^{-1}\cdot g)=A^{-1}(g^{-1})\cdot A^{-1}(g)=1$\,. Similarly,
		\[
		\left( (g,A)\bullet (h,B)\right)\bullet (h,B)^{-1}=(g,A)\,.
		\]
		Of course, because $\Oc^\ti$ is not associative the above smooth loop is not a Lie group.
		
	\end{example}
	Some information about smooth loops and octonion bundles can be found in \cite{Gri,Gri1}.
	
	\medskip	
	Loops that have only one-sided inverse are called \emph{left inverse loops} (resp. \emph{right inverse loops}). Left inverse loops appear naturally as algebraic structures
	on \emph{transversals} or \emph{sections} of a subgroup in a group. In this case the homogeneous structures are equipped with a binary operation. This observation, going
	back to R. Baer \cite{Bae} (cf. also \cite{Fo,KW}), lies at the heart of much current research on loops, also in smooth geometry and analysis.
	
	\begin{example}\label{tr} Let $G$ be a group with the unit $e$\,, $H$ be a subgroup, and $S\subset G$ be a left transversal to $H$ in $G$\,, i.e., $S$ contains exactly one point from each coset $gH$ in $G/H$. This means that any element $g\in G$ has a unique decomposition $g=sh$\,, where $s\in S$ and $h\in H$\,. This produces an identification $G=S\times H$ of sets. Let $p_S:G\to S$ be the projection on $S$ determined by this identification. If we assume that $e\in S$\,, then $S$ with the multiplication $$s\circ s'=p_S(ss')$$ and $e$ as the unit is a left inverse loop.
		
		Indeed, as $e\circ s=s\circ e=p_S(s)=s$\,, $e$ is the unit for this multiplication. For $a,b\in S$, there is $h\in H$ such that $p_S(a^{-1}b)=a^{-1}bh$\,. Hence,
		$$p_S(ap_S(a^{-1}b))=p_S(aa^{-1}bh)=p_S(bh)=b\,,$$
		that shows that $p_S(a^{-1}b)$ is a solution of the equation $a\circ x=b$\,. If $c,c'$ are two such solutions, then $p_S(ac)=p_S(ac')$\,, so there is $h\in H$ such that $ac=ac'h$\,, so $c=c'h$ and $c=c'$\,, since $S$ is transversal to $H$.
	\end{example}

	We will be interested with smooth (or analytic) structures on (inverse) loops. In this case $G$ is a smooth (analytic) manifold in which the multiplication (inverse) is a smooth (analytic) map.
	
	Exactly like it is done for Lie groups  we can define the corresponding local objects, \emph{smooth local loops}, in an obvious way.
	\begin{example}\label{e11}
		Let $G$ be a local smooth loop defined in a neighbourhood of $0$ in $\R^n$ with the product
		$$x\bullet y=(x+y+\frac{1}{2}[x,y])\,,$$
		where $[x,y]$ is a skew bilinear operation on $\R^n$.
		The unit is $e=0$ and the translations are local diffeomorphisms:
	$$l_x(y)=x+y+\frac{1}{2}[x,y]\,, \quad r_y(x)=x+y+\frac{1}{2}[x,y]\,.$$
		Every element $x$ possesses $-x$ such that $x\bullet(-x)=0$ but in general there is no inverse and no Inverse Property because
		$$(-x)\bullet (x\bullet y)=y-\frac{1}{4}[x,[x,y]]\,.$$
	If $[x,[x,y]]=0\,,$ then we deal with a nilpotent Lie group.
		
	\end{example}
	
	\section{Quasiloopoids and loopoids}

	\begin{question}
		We will look for the `right' concept of a smooth loopoid $G$ such that:
		\begin{itemize}
			\item it is in two ways a fibration over a submanifold  $M$ of `units", $\za,\zb:G\ra M$ which determine the set of composable elements;
			\item the translations are sufficiently regular;
			\item a smooth loopoid over one point is a smooth loop;
		\item it is as close as possible to the concept of a (Lie) groupoid;
			\item it has sufficiently rich variety of natural examples.
		\end{itemize}
	\end{question}
	
\noindent Our concept of a structure satisfying the first three requirements is the following.
We will consider weaker structures which we call (smooth) \emph{quasiloopoids} and the corresponding linear objects: \emph{skew algebroids}, with different versions of invertibility property. The first generalization of the concept of a Lie groupoid to nonassociative case is a \emph{smooth quasiloopoid}.
	\begin{definition} A \emph{smooth quasiloopoid} is a manifold $G$ equipped with surjective	
		submersions (target and source maps)  $\za,\zb:G\to M$ of $G$ onto the submanifold of units $M\subset G$ such that the product $m(g,h)=gh$ is defined if and only if $\zb(g)=\za(h)$\,. Moreover the multiplication $m:G^{(2)}\to G$\,, where $$G^{(2)}=\{(g,h)\in G\ti G\,|\, \zb(g)=\za(h)\}$$ is a smooth closed embedded submanifold of $G\ti G$\,. We also assume that $g\zb(g)=g$ and $\za(h)h=h$, i.e., elements of $M$ are units, and that the left and right translations maps
$$ l_g:\cF^\za(\zb(g))\to G\,,\ l_g(h)=gh\,, \quad\text{and}\quad r_h:\cF^\zb(\za(h))\to G\,,\ r_h(g)=gh\,,
$$
are injective immersions. By $\cF^\za(a)=\{ h\in G\,| \, \za(h)=a\}$ and $\cF^\zb(a)=\{ g\in G\,| \, \zb(g)=a\}$ we denote the target and source fibers ($\za$- and $\zb$-fibers), respectively. Sometimes we denote the obvious embedding $M\subset G$ as $\ze:M\to G$\,.
	\end{definition}
	
Now, let us assume that a quasiloopoid $G$ over $M$ with a partial multiplication $m$ and projections $\za,\zb:G\to M$\,, satisfies a very weak associativity condition, hereafter called \emph{unities associativity}:
	\be\label{ua} (xy)z=x(yz)\ \text{if one of}\ x,y,z\ \text{is a unit}\  (\text{i.e., belongs to}\ M)\,.
	\ee
	The above condition has to be understood as follows: if one side of equation (\ref{ua}) makes sense, the other makes sense and we have equality.
	The following proposition shows that the condition of unities associativity for a quasiloopoid over $M$ is rather strong and implies that the \emph{anchor map} $(\za,\zb):G\to M\ti M$ has nice properties, similar to these for Lie groupoids.
	\begin{proposition}
A quasiloopoid $G$ over $M$ satisfies the unities associativity condition if and only if		
		$$
		(\za,\zb):G\to M\ti M
		$$
		is a quasiloopoid morphism into the pair groupoid $M\ti M$, i.e.
		$$
		\za(gh)=\za(g)\ \text{and}\ \zb(gh)=\zb(h)\,.
		$$
The unities associativity assumption implies that the left and right translation maps can be viewed as injective immersions
\be\label{aga}{l}_{g}:\cF ^{\alpha}(\beta (g))\rightarrow \cF ^{\alpha }(\alpha  (g))\,, \quad
r_{h}:\cF ^{\beta }(\alpha(h))\rightarrow \cF ^{\beta }(\beta(h))\,.
\ee
	\end{proposition}
	\begin{proof}	
	
	According to (\ref{ua}), for  $(g,h)\in G^{(2)}$ we have
		$$\za(gh)(gh)=gh=(\za(g)g)h=\za(g)(gh)\,,$$
		so $\za(gh)=\za(g)$. Analogously we can prove $\zb(gh)=\zb(h)$\,.
	
\medskip\noindent
		Conversely, let $e\in M$ be such that $e(gh)$ makes sense. Then,
		$$e=\zb(e)=\za(gh)=\za(g)\,.$$
		Hence
		\be\label{a}e(gh)=\za(gh)(gh)=gh=(\za(g)g)h=(eg)h\,.
		\ee
		If this is $(eg)h$ that makes sense, then
		$$e=\zb(e)=\za(g)=\za(gh)$$ and we have (\ref{a}) again.
		Similarly we prove $(ge)h=g(eh)$ and $(gh)e=g(he)$\,.

\medskip\noindent As $\za(gh)=\za(g)$ and $\zb(gh)=\zb(h)$\,, it is easy to see that the images of $l_g$ and $r_h$ lie in $\cF ^{\alpha }(\alpha  (g))$ and $\cF ^{\beta }(\zb(h))$\,, respectively.
	
\end{proof}
	\begin{definition}
		A quasiloopoid satisfying the unities associativity assumption  and such that the maps (\ref{aga}) are diffeomorphisms will be called a \emph{loopoid}. A (quasi)loopoid $G$ over the manifold $M$ will be denoted by
	$G \rightrightarrows M$\,.
	\end{definition}
	\begin{remark}
		In a loop, the multiplication is globally defined, so the unity associativity is always satisfied
		by properties of the unity element. In this sense, smooth loops are smooth loopoids over one point. It is not true for quasigroups, and they are not quasiloopoids over a point. In fact the conditions $g\zb(g)=g$ and $\za(h)h=h$ require the existence of the two sided identity.
	\end{remark}
\begin{remark}
			Smooth Moufang loops, considered by Mal’cev \cite{mal}, are particular cases of loops which satisfy any of the three following equivalent conditions
		\[
			((ax)a)y = a(x(ay))\,,\quad	((xa)y)a = x(a(ya))\,, \quad(ax)(ya)= (a(xy))a\,.
			\]
			They can be considered as loopoids over a single point, because the above multiplications are clearly well defined when the source and target maps are the same (projections onto the unit).
			
Note that when we speak about a general loopoid, the above multiplications are generally not well defined, so loopoids do not seem to be proper generalizations of Moufang loops except for smooth Moufang loops themselves.
	\end{remark}
	The definition of a morphisms $\Phi:G_1\to G_2$ between two (quasi)loopoids is completely clear: we impose
	that $\Phi(xy)=\Phi(x)\Phi(y)$. In particular, both sides have the meaning simultaneously. It follows that $\Phi$ induces a map of the sets of units $\phi:M_1\to M_2$ and intertwines the target and source maps. In this way we obtain two categories for which, as we will see, one can construct a `Lie functor' to the categories of skew-algebroids and almost-Lie algebroids.  If $\Phi$ is inclusion of the submanifold, we call $G_1$ a \emph{(quasi)subloopoid} of $G_2$.
	
	The elements of $G ^{(2)}$ are referred to as \emph{composable} (or \emph{admissible}) pairs. Note that the full information about the loopoid is contained in the
	\emph{multiplication relation} which is a subset $G^{(3)}\subset G\ti G\ti G$,
	\be\label{mr}
	G^{(3)}=\left\{(x,y,z)\in G\ti G\ti G\, |\ (x,y)\in G^{(2)}\ \text{and}\ z=xy\right\}\,.
	\ee Of course, for loops the multiplication is globally defined and $\za,\zb:G\to\{e\}$\,, so that $G^{(2)}=G^2=G\ti G$\,.
	On a (quasi)loopoid we can consider inverse properties.
	
	\begin{definition} A smooth \emph{inverse} in a quasiloopoid $G \rightrightarrows M$ is an
		inversion mapping (diffeomorphism) $\iota :G \rightarrow G$\,,
		satisfying the following properties (where we write  $g^{-1}$ for $\iota (g)$)\,:
		$$
		g^{-1}(gh) =h\quad \text{and}\quad (hg)g^{-1}=h\,.$$
		We say also the $G$ has the \emph{Inverse Property} (I.P.).
	\end{definition}
	\begin{theorem}
		A quasiloopoid with the \emph{Inverse Property} is a loopoid.
	\end{theorem}
	\begin{proof}
		The existence of inverses implies $\za(gh)=\za(g)$ and $\zb(gh)=\zb(h)$\,. Indeed,
		$g^{-1}(gh)$ is composable, hence $\za(g)=\zb(g^{-1})=\za(gh)$ and similarly for $\beta$.
		The inverses assure that translations are diffeomorphisms between the corresponding $\za$ and $\zb$-fibers.
		
	\end{proof}
\noindent We will call this kind of loopoids \emph{inverse loopoids} or \emph{I.P. loopoids}. Like for the loops one can show that for I.P loopoids we have the identities
$$gg^{-1} = \za(g)\,,\quad g^{-1}g = \zb(g)\,,\quad (g^{-1})^{-1} = g\,,\quad \text{and}\quad  (gh)^{-1} = h^{-1} g^{-1}\,.$$
We can also consider only left/right inverses and speak about left or right inverse quasiloopoids. For instance, left/right transversals are left/right inverse quasiloopoids. The variant that we choose will depend on particular examples.
	
	\begin{example}\label{exa}
		Let $G'$ be a smooth loop with the unit $e$ and let $N$ be a manifold. On $G=G'\ti N\ti N$ we have an obvious structure of a smooth loopoid as a product structure of the loop $G'$ and the pair groupoid $N\ti N$ over $M= \{(e,s,s)\,|\, s\in N\}\subset G$\,. The anchor map is
		$$\za(x,s,t)=(e,s,s)\,,\quad \zb(x,s,t)=(e,t,t)$$
		and the partial multiplication reads
		$$(x,s,t)\bullet(y,t,r)=(xy,s,r)\,.$$
		If $G'$ is an inverse loop, then $G$ is an I.P. loopoid with the inverse $\zi(x,s,t)=(x^{-1},t,s)$\,. In this example $G'$ can be identified with the \emph{isotropy loop} $G_u$ at each $u\in M$\,.
	\end{example}

	\begin{remark}\label{prolong}
	Let $\pi: P \to M$ be a fibration, that is a surjective submersion. The prolongation of a smooth quasiloopoid (loopoid) $G\rightrightarrows M$ over $\pi$\,, defined as the set
		\[
		{\mathcal P}^{\pi}G = P \mbox{$\;$}_\pi \kern-3pt\times_\alpha G
		\mbox{$\;$}_\beta \kern-3pt\times_\pi P = \{ (p, g, p') \in P
		\times G \times P / \pi(p) = \alpha(g), \; \; \beta(g) = \pi(p')
		\}\,.
		\]
		is a smooth loopoid over $P$ with structural maps given by
	\[
\za^{\pi}(p,g,p')=p\,, \quad 	\zb^{\pi}(p,g,p')=p'\,,\quad \ze^{\pi}(p)=(p,\ze(\pi(p)),p)\,,
		\]
		and the multiplication $(p,g,p')\bullet(p',h,p'')=(p,gh,p'')$ where $(g,h)\in G^{(2)}$. If $G$ is an inverse loopoid, then $\mathcal P ^{\pi}G$ is an inverse loopoid with the inversion map $\zi(p,g,p')=(p',g^{-1},p)$\,. Therefore, we can produce a loopoid out of a quasiloopid by the prolongation of that over a fibration.
	\end{remark}
\begin{example}\label{almost}
		A more complicated example we can construct as follows.
		Consider the pair groupoid $\mathcal{G}=\R^2\ti\R^2$
		with the standard target and source maps
		$\za(u,v)=u$\,, $\zb(u,v)=v$
		and composition $(u,v)(v,z)=(u,z)$\,, $u,v,z\in\R^2$.
		For a diffeomorphism $\zf:\R\to\R$ being an odd function, $\zf(-x)=-\zf(x)$\,, define a submanifold
		\be\label{GG}G=\left\{\left((a_1,b_1),(a_2,b_2)\right)\in\mathcal{G}: a_1-a_2=\zf(b_1-b_2)\right\}\,.
		\ee
		It is a quasiloopoid, with the target and source maps inherited from $\mathcal{G}$,
		and the partial multiplication
		$$\left((a_1,b_1),(a_2,b_2)\right)\bullet\left((a_2,b_2),(a_3,b_3)\right)=
		\left((a_1,b_1),(a_1+\zf(b_3-b_1),b_3)\right)\,.
		$$
		Indeed, for fixed $\left((a_1,b_1),(a_2,b_2)\right)\in\R^2\ti\R^2$ the left unit is $(a_1,b_1)$ and the right one is $(a_2,b_2)$\,.
		For $G$ we have $\za(gh)=\za(g)$\,, but generally, for non-linear $\zf$\,, we have $\zb(gh)\ne\zb(h)$\,, so $G$ is not an inverse loopoid. However, it is interesting that $G$ has a left inverse $$\zi_l((a_1,b_1),(a_2,b_2))=((a_2,b_2),(a_1,b_1))\,.$$
Indeed,
		\beas&&\left((a_2,b_2),(a_1,b_1)\right)\bullet\left(\left((a_1,b_1),(a_2,b_2)\right)
		\bullet\left((a_2,b_2),(a_3,b_3)\right)\right)\\
		&=&\left((a_2,b_2),(a_1,b_1)\right)\bullet\left((a_1,b_1),(a_1+\zf(b_3-b_1),b_3)\right)\\
		&=&\left((a_2,b_2),(a_2+\zf(b_3-b_2),b_3)\right)
		=\left((a_2,b_2),(a_3,b_3)\right)\,.
		\eeas
		The last equality follows from the fact that $\left((a_2,b_2),(a_3,b_3)\right)\in G$\,, so by (\ref{GG})
		$$a_2-a_3=\phi(b_2-b_3)=-\phi(b_3-b_2)\,.$$
		All this implies that $G$ is a left inverse quasiloopoid.
		\end{example}
	
	\begin{example}
			 If the anchor map of a loopoid is diagonal (trivial), i.e., $\alpha=\beta$, then the loopoid is actually a bundle of loops. This bundle may be locally trivial as a fiber bundle and not as a loop
			 bundle,  since a local trivialization (in the general sense of fiber bundles) is not necessarily a local trivialization in the sense of loop bundles.  The loops for different fibers could be even non-isomorphic as loops. Conversely, any bundle of loops, i.e., a surjective submersion $\tau: G \rightarrow M$ whose each fiber carries a loop structure, is canonically a loopoid over $M$. In this case $M$ is embedded in $G$ as the section of unit elements in each loop fiber, $\alpha=\beta=\tau$, and the partial multiplication $k=g\cdot h$ is defined if and only if $\tau (k)=\tau (g)=\tau (h)$ and is $k=gh$ -- the loop multiplication in each fiber \cite{Gr}.
		 \end{example}

\subsection{Bisections and isotropy loops}
	The concept of global bisections for loopoids can be defined as for Lie groupoids. A submanifold $\mathcal B$ of loopoid $G\rightrightarrows M$ is called a \emph{bisection} of $G$ if the restricted maps of the source and the target maps $\za|_{\mathcal B}$ and $\zb|_{\mathcal B}$ are both diffeomorphisms. It means that for any bisection $\mathcal B$\,, there are corresponding $\za$-section $\tau:=(\za|_{\mathcal B})^{-1}:M\to \mathcal B$ and $\zb$-section $\sigma:=(\zb|_{\mathcal B})^{-1}:M\to \mathcal B$ such that $\zb \circ \tau$ and $\za \circ \sigma$ are diffeomorphism maps on $M$ to itself, and $\za \circ \sigma=(\zb \circ \tau)^{-1}$. For each bisection $\mathcal B$ we can define a left translation $l_{\mathcal B}:G\to G$ and a right translation $r_{\mathcal B}:G\to G$ on the loopoid, given by
	\[
	l_{\mathcal B}(g)=\sigma(\za(g))g\,, \quad r_{\mathcal B}(g)=g\tau(\zb (g))\;.
	\]
	Since the multiplication of two bisections $\mathcal B_1\mathcal B_2=l_{\mathcal B_1}\mathcal B_2=r_{\mathcal B_2}\mathcal B_1$ is again a bisection, the bisections of an inverse loopoid $G$ form an inverse loop with the obvious unit, $\mathcal B(m)=\ze(m)$\,, and inverse $\mathcal B^{-1}(m)=\mathcal B(m)^{-1}$\,.
	
\medskip\noindent In the quasiloopoid context, since the left and the right translations are not in general diffeomorphisms, therefore generally there is no global bisection. However, we can define local bisections on quasiloopoids as follow.
	A submanifold $\mathcal B$ of a quasiloopoid $G$ is a \emph{local bisection} if the restricted maps $\za|_{\mathcal B}$ and $\zb|_{\mathcal B}$ are local diffeomorphisms onto open sets $U,V\in M$\,. Consequently, there exist a local $\za$-section $\tau:U\to \mathcal B$ and a local $\zb$-section $\sigma:V\to \mathcal B$ such that $\zb\circ \tau:U\to V$  and $\za\circ \sigma:V\to U$ are diffeomorphisms. Using local bisection $\mathcal B$, we can define locally a left translation $l_{\mathcal B}:\za^{-1}(V)\to \za^{-1}(U)$ and a right translation $r_{\mathcal B}:\zb^{-1}(U)\to \za^{-1}(V)$\,, given by
	\[
	l_{\mathcal B}(g)=\sigma(\za(g))g\,, \quad r_{\mathcal B}(g)=g\tau(\zb (g))\;.
	\]	
	
\noindent Note that there is an equivalent definition of a bisection (or local bisection) \cite{Ma}, which takes a bisection (or local bisection) to be the pair of maps $\tau$ and $\sigma$, rather than submanifold $\mathcal B$. We will use both equivalent definitions.

\noindent The following Lemma can be borrowed from the theory of Lie groupoids (see \cite[Proposition 3.7]{Mei})
\begin{lemma} For any inverse loopoid $G\rightrightarrows M$ and any $m\in M$, the restriction of $\za$ to
the source fiber $\zb^{-1}(m)$ has constant rank.
\end{lemma}
\begin{proof}
To show that the ranks of
$$\za|_{\zb^{-1}(m)}:\zb^{-1}(m)\to M$$
coincide at given points $g,g'\in \zb^{-1}(m)$, let $\zs$ be a local bisection containing $g'g^{-1}$ and let $U=\zb(\zs)$\,, $V=\za(\zs)$\,. We have an obvious diffeomorphism $\Phi_\zs:U\to V$ and
$$\Psi_\zs:\za^{-1}(U)\to\za^{-1}(V)\,,\quad h\mapsto uh\,,$$
where $u$ is the unique element in $\zs$ such that $\zb(u)=\za(h)\;.$ It is easy to see that $\Psi_\zs(g)=g'$ and $\zb\circ\Psi_\zs=\zb$\,, so $\Psi_\zs$ restricts to a diffeomorphism of each $\zb$-fiber. In the commutative diagram
$$\xymatrix{
\za^{-1}(U)\cap\zb^{-1}(m) \ar[d]^{\za|_{\zb^{-1}(m)}}\ar[rr]^{\Psi_\zs|_{\zb^{-1}(m)}} && \za^{-1}(V)\cap\zb^{-1}(m)\ar[d]^{{\za|_{\zb^{-1}(m)}}} \\
U\ar[rr]^{\Phi_\zs} && V }
$$
the horizontal maps are diffeomorphisms and the upper map takes $g$ to $g'$. Hence the ranks of the vertical map at $g$ and $g'$ coincide.
\end{proof}
	\begin{theorem}  For every smooth inverse quasiloopoid $G\rightrightarrows M$ the intersection $G_m$ of  $\cF^\za(m)$ and $\cF^\zb(m)$ is an inverse loop (\emph{isotropy loop}) for all $m\in M\;.$
	\end{theorem}
	\begin{proof} It is easy to see that $G_m$ is closed with respect to the multiplication and inverse inherited from $G$, so it is an inverse loop. It is a submanifold of $G$ according to the above Lemma, so it is smooth.
	\end{proof}

	\section{Skew algebroids}
	
	Let $\zt:E\to M$ be a vector bundle of rank $n$ over an
	$m$-dimensional manifold $M$ and let $\zp:E^\*\ra M$ be its dual.
	Recall that the Grassmann algebra $\Gr(E)=\oplus_{i=0}^\infty\Sec(\we^i
	E)$
	of multisections of $E$ is a
	graded commutative associative algebra with respect to the wedge
	product.
	We use affine coordinates
	$(x^a,\zx_i)$ on $E^\*$ and the dual coordinates $(x^a,y^i)$ on
	$E$\,, associated with dual local bases, $(e_i)$ and $(e^i)$\,, of
	sections of $E$ and $E^\*$\,, respectively.
	
	\begin{definition}  A \emph{skew algebroid} structure on $E$ is given by a linear bivector
		field $\zP$ on $E^\*$ (linear skew \emph{Leibniz structure}). In local coordinates,
		$$ \Pi =\frac{1}{2}c^k_{ij}(x)\zx_k
		\partial _{\zx_i}\we \partial _{\zx_j} + \zr^b_i(x) \partial _{\zx_i}
		\wedge \partial _{x^b}\,,
		$$ where $c^k_{ij}(x)=-c^k_{ji}(x)\;.$ If \ $\Pi$ is a Poisson tensor, we speak about a \emph{Lie algebroid}.
	\end{definition}
	
	As the bivector field $\Pi$ defines a bilinear bracket
	$\{\cdot,\cdot\}^\Pi$ on the algebra $C^\infty(E^\*)$ of smooth
	functions on $E^\*$ by
	$\{\zvf,\psi\}^{\zP}=\langle\zP,\xd\zvf\we\xd\psi\rangle$\,, where
	$\langle\cdot,\cdot\rangle$ stands for the contraction.
	
	\begin{theorem}\label{ty}
		A skew algebroid structure $(E,\Pi)$ can be equivalently defined
		as a skew-symmetric $\R$-bilinear bracket $[\cdot ,\cdot]^\Pi $
		on the module $\Sec(E)$ of sections of $E$, together with a vector
		bundle morphisms\ $\zr=\zr^\Pi \colon E\rightarrow T M$ (\emph{the
			anchor}), such that
		$$ [X,fY]^\Pi =\zr^\Pi(X)(f)Y +f [X,Y]^\Pi\,,
		$$ for all $f \in C^\infty (M)$\,, $X,Y\in \Sec(E)$\,.
	\end{theorem}
	
	The bracket $[\cdot,\cdot]^\Pi$ and the anchor $\zr^\Pi$ are related
	to the bracket $\{\cdot,\cdot\}^{\zP}$ according to the formulae:
	\beas
	\zi([X,Y]^\Pi)&= \{\zi(X), \zi(Y)\}^{\zP}\,,  \\
	\zp^\*(\zr^\Pi(X)(f))  &= \{\zi(X), \zp^\*f\}^{\zP}\,,
	\eeas
	where we denoted with $\zi(X)$ the linear function on $E^\*$
	associated with the section $X$ of $E$, i.e., $\zi(X)(e^\ast_p)=\la
	X(p),e^\ast_p\ran$ for each $e^\ast_p\in E^\ast_p$\,.
	
\noindent The Lie algebroid bracket satisfies the Jacobi identity
	$$
	[[X,Y]^\Pi,Z]^\Pi=[X,[Y,Z]^\Pi]^\Pi-[Y,[X,Z]^\Pi]^\Pi\,.
	$$
	It follows that the anchor map is a Lie algebra homomorphism
	\be\label{AH}
	\zr^\Pi([X,Y]^\Pi)=[\zr(X),\zr(Y)]\,,
	\ee
	where the second bracket is the bracket of vector fields. If we assume (\ref{AH}) without the
	Jacobi identity, then we deal with an \emph{almost Lie algebroid}.
\begin{remark}
	The tangent algebra of a smooth loop is	a skew algebra, i.e., a real
	vector space equipped with a bilinear skew operation. Actually, skew algebras and skew algebroids over a point are the same things.		
		\end{remark}
\begin{remark}
		The tangent algebra of a smooth Moufang loop is a  Mal'cev algebra which is a skew algebra satisfying \cite{mal}
		\[
		\left[ [X,Y],[X,Z]\right] =\left[ \left[ [X,Y],Z\right] ,X\right] +\left[ \left[ [Y,Z],X\right] ,X\right]+\left[ \left[ [Z,X],X\right] ,Y\right]\,,
		\]
	for every $X,Y,Z$. One can see that Mal'cev algebras are skew algebroids over a point satisfying the above identity.
\end{remark}

	\begin{example}
		Any skew algebra bracket $[\cdot,\cdot]_0$ on $\R^n$ gives rise to an almost Lie algebroid
		$E=\T M\ti\R^n\to M$ with the bracket
		$$[(X,v),(Y,w)]=\left([X,Y],X(w)-Y(v)+[v,w]_0\right)$$
		and the anchor $\zr(X,v)=X$\,.
		Here, $X,Y$ are vector fields on $M$ and $v,w$ are functions on $M$ with values in $\R^n$.
	\end{example}

We can reduce the prolongation of a Lie algebroid over a smooth map to a skew algebroid over a smooth map (\cite{HiMa}, \cite{LeMa}). In the following we will consider the process for the prolongation of a skew algebroid over a submersion. As we expect from the Remark \ref{prolong}, we will see that the prolongation of a skew algebroid over a fibration is an almost Lie algebroid.
\begin{remark}\label{pro}
The prolongation of the skew algebroid $(E, \lcf\cdot,\cdot\rcf,\zr)$ over the submersion $\pi:P\to M$ which is defined by
\[
\mathcal P^\pi E=\{(X,V)\in E\times \sT P|\quad \zr(X)=\sT \pi(V)\}\,,
\]
where $\sT\pi:\sT P \to \sT M$ is the tangent map to $\pi$\,, is equipped with an almost Lie algebroid structure.
\end{remark}
\begin{proof}
The space $\mathcal P^\pi(E)$ is a vector bundle over $P$ with vector bundle projection
\[
\tau ^{\pi}:\mathcal P^\pi E \to P\,,\quad \tau ^{\pi}(X,V)\to \tau_{P}(V)\,,
\]
where $\tau_{P}:\sT P\to P$ is the canonical projection. In fact, for $p\in P$ every fiber
\[
(\tau ^{\pi})^{(-1)}_{p}=(\mathcal P^{\pi}E)_p=\{(X, V)\in E_{\pi(p)}\times \sT _p P|\quad \zr(X)=(\sT _p\pi)(V) \}
\]
where $E_{\pi(p)}$ is the fiber of $E$ over the point $\pi(p)\in M$, is a finite-dimensional real vector subspace of $E_{\pi(p)}\times \sT _p P$ because
 $$
 dim(\mathcal P^{\pi} E)_p= rank (E)+dim(P)-dim\left( \zr(E_{\pi(p)})+(\sT_p \pi)(\sT_p P)\right),
 $$
and since $\pi$ is a submersion, we have
\[
 \zr(E_{\pi(p)})+(\sT_p \pi)(\sT_p P)=\sT _{\pi(p)}M\,,\quad \mbox{for all} \quad p\in P\,,
\]
which implies that there exists $c\in \mathbb N$ such that $dim\left( \zr(E_{\pi(p)})+(\sT_p \pi)(\sT_p P)\right)=c$ and therefore $\mathcal P^{\pi} E$ has constant rank.

To know how the skew algebroid structure on $\mathcal P^\pi E$  can be characterized, we need to describe the sections of the vector bundle $\tau ^{\pi}:\mathcal P^\pi E \to P$ and for that we use the pullback vector bundle $E$ over $\pi$. The pullback of the $E$ over $\pi$
\[
\pi^*E=\{(p,X)\in P\times E|\quad \pi(p)=\tau(X)\}\,,
\]
is a vector bundle over $P$ with vector bundle projection $$pr_1:\pi^*E\to P\,,\quad pr_1(p,X)=p\,,$$
and any section $X$ of $E$ over $M$ induces a section of  $\pi^*E$ over $P$, simply by defining
\begin{equation}\label{sigma}
\sigma=f_i(X_i\circ \pi)\,, \quad f_i\in C^{\infty}(P)\,,\quad X_i\in \Sec E\,.
\end{equation}
We have to go one step further and take the pullback of
$$\xymatrix{
\mathcal P^{\pi}E \ar[d]^{\zr ^{\pi}}\ar[rr]^{\zc^*} && \pi^*E \ar[d]^{\pi^*{\zr}} \\
	\sT P\ar[rr]^{d\pi} && \pi^*(\sT M)}
$$
in the category of vector bundles over $P$. Note that we denote the map $\mathcal P^{\pi}E \to E$ by $\zc$.
Then, sections of the vector bundle $\tau ^{\pi}:\mathcal P^{\pi}E\to P$ are of the form $X'\oplus \sigma$, where $X'\in \Sec(\sT P)$ and $\sigma\in \Sec(\pi^*E)$ and
\[
d\pi(X')=\pi^*\zr(\sigma)\,,
\]
and so from (\ref{sigma}), for $X^{\wedge}\in \Sec(\mathcal P^{\pi}E)$ we have $X^{\wedge}(p)=(\sigma(p),X'(p))$ with $$(\sT_p\pi)(X'(p))=f_i(p)\zr(X_i)(\pi(p))\,.$$

The vector bundle $\tau^{\pi}:\mathcal P^{\pi}E\to P$ is equipped with a skew algebroid structure with the anchor map $\zr^{\pi}(X^{\wedge})=X'$\,. The bracket of two sections $(f_i(X_i\circ \pi),X')$ and $(g_j(Y_j\circ \pi),Y')$ of $\tau^{\pi}:\mathcal P^{\pi}E\to P$ with $f_i,g_j\in C^{\infty}(P)$\,, $X_i,Y_j\in Sec(E)$ and $X',Y'\in Sec(\sT P)$\,, is
\[
\begin{array}{rcl}
\lcf (f_i(X_i\circ \pi),X'), (g_j(Y_j\circ \pi),Y') \rcf ^{\pi}&=&((f_ig_j\lcf X_i,Y_j\rcf \circ \pi)+X'(g_j)(Y_j \circ \pi)\\[4pt]
&&-Y'(f_i)(X_j \circ \pi),[X',Y'])\,.
\end{array}
\]
One may prove that the bracket $\lcf \cdot, \cdot \rcf^{\pi}$ on $Sec(\mathcal P^{\pi}E)$ satisfies Jacobi identity if and only if the bracket $\lcf \cdot, \cdot \rcf$ on $Sec(E)$ satisfies Jacobi identity. We see that
\[
\zr^{\pi}\lcf X^{\wedge},Y^{\wedge}\rcf^{\pi} =[\zr^{\pi}(X^{\wedge}),\zr^{\pi}(Y^{\wedge})]\,,\quad\mbox{for all}\quad  X^{\wedge},Y^{\wedge} \in Sec(\mathcal P^{\pi}E)\,,
\]
and does not depends on the identity $\zr\lcf X,Y\rcf =[\zr(X),\zr(Y)]$ of sections of $E$. Therefore, the prolongation of a skew algebroid over a submersion is equiped with an almost Lie algebroid structure.

\end{proof}
\begin{definition}
		A subbundle $B\to N$ of a skew algebroid $E\to M$ with anchor $\zr:E\to\T M$ is called a \emph{skew subalgebroid} if
		$\zr(B)\subset \T N$ and
		$$\Sec(E,B)=\{ X\in\Sec(E)\,|\,X_{|N}\in\Sec(B)\}$$ is closed with respect to the bracket on $\Sec(E)$\,.
\end{definition}
\begin{example} If $N$ is a submanifold of $M$, then $\T N$ is canonically a Lie subalgebroid of $\T M$.
\end{example}
	Exactly like for Lie algebroids we have the following.
\begin{proposition} If $B\to N$ is a skew subalgebroid of $E\to M$\,, then it inherits a unique skew algebroid structure such that the natural map $\Sec(E,B)\to\Sec(B)$ preserves the bracket. If $E$ is almost Lie, then $B$ is almost Lie.
\end{proposition}
	Exactly like for Lie algebroids there is a canonical construction of the direct product $(E_1\ti E_2)\to (M_1\ti M_2)$ of skew (almost Lie) algebroids $E_i\to M_i$\,, $i=1,2$\,, which leads to the concept of a skew algebroid morphism.
\begin{definition}(cf. \cite{Gr0})
		If $E_i\to M_i$\,, $i=1,2$\,, are skew (almost Lie ) algebroids, then a \emph{skew (almost Lie) algebroid morphism} between them is a vector bundle morphism $\zvf:E_1\to E_2$ such that its graph $$\Graph(\zvf)\subset E_1\ti E_2$$ is a skew (almost Lie) subalgebroid of the direct product $E_1\ti E_2$\,.
\end{definition}	
	
\begin{example}
In the Remark \ref{pro}, the pair $(pr_1,\pi)$ is a morphism between the skew algebroids $(\mathcal P^{\pi} E,\lcf \cdot,\cdot\rcf^{\pi},\zr^{\pi})$ and $(E,\lcf \cdot,\cdot\rcf,\zr)$\,, where $\pr_1:\mathcal P ^{\pi} E \to E$ is the cononical projection on the first factor.
\end{example}

	\section{Generalized Lie functor}
	
	In the context of the Lie algebroid theory, the `nonassociative' objects have already been studied by us and other authors
		\cite{GU1,GU2,GG,GGU,GJ}.
		The main objects are \emph{skew algebroids}, i.e, vector bundles $E$ with a skew-symmetric bracket operation on sections which does not satisfy the Jacobi identity but admits an anchor map $a:E\to \T M$\,. If the anchor is a bracket
		homomorphism, we speak about \emph{almost Lie algebroids}. The homotopies of admissible paths in almost Lie algebroids are well defined \cite{GG,GJ}. The importance of such objects comes from the fact that they form a nice framework for generalizations of Lagrangian and Hamiltonian mechanics and nonholonomic systems \cite{GGU,GLMM}. We show that we can obtain skew and almost Lie algebroids as infinitesimal parts of smooth loopoids and that some of the standard geometric constructions for the tangent and cotangent groupoid for a given Lie groupoid can be extended to this category.
	
	The infinitesimal part of a Lie groupoid $G\rightrightarrows M$ is the normal bundle $AG=\zn(G,M)\to M$\,,
	$$\zn(G,M)=\T G_{|M}/\T M$$ equipped with the natural Lie algebroid bracket. The association $G\mapsto AG$ is a functor from the category of Lie groupoids to the category of Lie algebroids; we will refer to it as the \emph{Lie functor}. To any $X\in \Sec(AG)$ it corresponds the unique left-invariant vector field $\lvec{X}$ on $G$ such that, for each $m\in M$ the vector $\lvec{X}(m)$  represents $X_m$ in $\zn(G,M)$\,. Similarly we obtain the right-invariant vector fields $\rvec{X}$\,.
	
	Left invariant vector fields are closed with respect to the Lie bracket and define a Lie algebroid bracket $[\cdot,\cdot]_l$ on $AG$ such that
	$$ \lvec{[X,Y]_l}=[\lvec{X},\lvec{Y}]\,.
	$$
	Similarly, for the right-invariant vector fields
	$$ \rvec{[X,Y]_r}=[\rvec{X},\rvec{Y}]\,.
	$$
 In a Lie groupoid, left translations commute with right translations, i.e,
	$$ [\rvec{X},\lvec{Y}]=0\,.
	$$
	Moreover, $[X,Y]_l=-[X,Y]_r$\,.

\bigskip
We will try to follow these construction of the bracket on $\Sec(AG)$ for any (local) smooth quasiloopoid $G\rightrightarrows M$\,.

\begin{definition}	
	For $X\in\Sec(AG)$ we define its \emph{left prolongation}  (\emph{left fundamental vector field on $G$}) by
	
	$$\lvec{X}(g)=(\T_{\zb(g)}l_g)(\bar{X}^\za(\zb(g))\,.$$
\end{definition}	
	Here, $\bar{X}^\za(a)\in\T\cF^\za(a)$ represents $X(a)$ in the normal bundle $AG=\zn(G,M)$. This is a canonical isomorphism of the vector bundles $AG\to M$ and $\T\cF^\za_{|M}\to M$.
	A similar construction we do for the \emph{right prolongation} (\emph{right fundamental vector field on $G$})
	
	$$\rvec{X}(g)=(\T_{\za(g)}r_g)(\bar{X}^\zb(\za(g))\,,$$
	
	where $\bar{X}^\zb(a)\in\T\cF^\zb(a)$ represents $X(a)$ in the normal bundle $\zn(G,M)$\,.
	\begin{theorem}
		For any $X\in\Sec(AG)$ the vector fields $\lvec{X}$ and $\rvec{X}$ are smooth on $G$\,.
	\end{theorem}
	\begin{proof}
		Let $f$ be a smooth function on $G$. We will show that $\lvec{X}(f)$ is smooth. The normal bundle $AG$ and the tangent bundle of $\za$-fibration along $M$ are canonically isomorphic and we can take a smooth vector field $Y$ on $G$ which  is tangent to $\za$-fibers and extends $\bar{X}^\za$ on $M$\,. The latter is a smooth section of the subbundle of tangent spaces to $\za$-fibration along $M$\,. We have
		$$\lvec{X}(f)(g)=(\sT_{\zb(g)}l_g)\left(\bar{X}^\za(\zb(g)\right)(f)=(\sT_{\zb(g)}l_g)\left(Y(\zb(g)\right)(f)=Y(f\circ l_g)(\zb(g))\,,$$
		where $f\circ l_g:\cF^\za(\zb(g))\to \R$\,.
		But
		$$Y(f\circ l_g)(\zb(g))=(0,Y)(f\circ m)(g,\zb(g))$$
		which clearly depends smoothly on $g$. Here $(0,Y)$ is a self-explaining notation for a vector field on $G\ti G$ and $m$ is the quasiloopoid multiplication. Note that the vector field $(0,Y)$ on $G\ti G$ is really tangent to $G^{(2)}$
		as $\sT\za (Y)=0=\sT\zb (0)$\,.  One can also follow the proof of Theorem 2.1 in \cite{Kub}.
	\end{proof}
 The fact that the left and right translations are immersions implies the following.
\begin{proposition}\label{preg}
For any smooth quasiloopoid $G$ the maps $\lvec{X}(\zb(g))\mapsto \lvec{X}(g)$ and $\rvec{X}(\za(h))\mapsto \rvec{X}(h)$ are linear isomorphisms of vector spaces. In other words, for any $g$, the vector spaces $\{\lvec{X}(g)\,|\, X\in A_{\zb(g)}G\}$ and $\{\rvec{X}(h)\,|\, X\in A_{\za(h)}G\}$ are canonically isomorphic with
the fibers $A_{\zb(g)}G$ and $A_{\za(h)}G$\,, respectively.
\end{proposition}
	In the case of quasiloopoids, the left (right) prolongations are no longer left (right) invariant and are not closed with respect to the Lie bracket of vector fields on $G$, as $G$ is not associative. Also the left and right prolongations do not commute.
	However, we still can define two brackets, $[\cdot,\cdot]_l$ and $[\cdot,\cdot]_r$\,, on sections of $AG$ by putting
	$$
	{[X,Y]}_l(a)=[\lvec{X},\lvec{Y}](a)\,,\quad {[X,Y]}_r(a)=[\rvec{X},\rvec{Y}](a)\,,\ a\in M\,.
	$$
	We will denote $AG$ equipped with these brackets by $A^lG$ and $A^rG$, respectively.
	The fundamental fact is the following.
	\begin{theorem}
		The above brackets are skew algebroid brackets on $AG$ with the anchors $\zr_l(X)=\T\zb(\bar X^\za)$ and $\zr_r(X)=\T\za(\bar X^\zb)$\,. We have additionally $\zr_r=-\zr_l$\,.
	\end{theorem}
	\begin{proof}
		The brackets are clearly skew-symmetric and $\R$-bilinear, as $\lvec{cX}=c\lvec{X}$, $\rvec{cX}=c\rvec{X}$ for $c\in\R$. More generally, if $f$ is a smooth function on $M$, then $\lvec{fX}=f^\zb\lvec{X}$, where $f^\zb=f\circ\zb$ is the $\zb$-pull-back of $f$ to $G$.
		In consequence,
		$$
		{[X,fY]}_l^\za(a)=[\lvec{X},\lvec{fY}](a)=[\lvec{X},f^\zb\lvec{Y}](a)=f(a)[\lvec{X},\lvec{Y}](a)+\lvec{X}(f^\zb)(a)\bar{Y}^\za(a)\,.
		$$
		As
		$$\lvec{X}(f^\zb)(a)=\lvec{X}(f\circ\zb)(a)=\T_a\zb(\bar{X}^\alpha)(f)=\zr_l(X)(f)(a)\,,$$
		we get
		$$
		{[X,fY]}_l=f[X,Y]_l+\zr_l(X)(f)Y\,,
		$$
 which shows that $\zr_l$ is the anchor for the left bracket. Similarly for the right bracket.
		For $X\in AG$  we put
$$\zr(X)=\bar{X}^{\za}-\bar{X}^{\zb}\,.
$$
As $\bar{X}^{\za}$ and $\bar{X}^{\zb}$ represent the same normal vector to $M$, we have $\zr(X)\in\T M$. Then,
		$$\zr_l(X)=\T\zb(\bar X^\za)=\T\zb\left(\bar X^\za-\bar{X}^\zb\right)=\zr(X)\,.$$
		Similarly we show that
		$$\zr_r(X)=\T\za(\bar X^\zb)=\T\za\left(\bar X^\zb-\bar{X}^\za\right)=-\zr(X)\,.$$
	\end{proof}
	As the two anchors differ only by sign, we will call the left anchor just an anchor and denote it by $\zr$.

 \begin{theorem}
		For a smooth local loop $G$ the left and the right skew infinitesimal algebras differ by sign, $[X,Y]_l=[\lvec{X},\lvec{Y}](e)=-[\rvec{X},\rvec{Y}](e)=-[X,Y]_r$\,. Moreover, $[\lvec{X},\rvec{Y}](e)=0$ for all $X,Y\in AG$\,.
	\end{theorem}
	\begin{proof}
		Consider a local loop multiplication in a neighborhood of $0$ in $\R^n$ with $0$ as the neutral element with the Taylor expansion around $0$
		$$(x\bullet y)^k=a^k+b^k_ix^i+d^k_jy^j+c^k_{ij}x^iy^j+e^k_{ij}x^ix^j+f^k_{ij}y^iy^j+o(\nm{x}^2)+o(\nm{y}^2)+ o(|\bra{x}{y}\ran|)\,.$$
		According to the fact that $0$ is the unit, we have $a^k=0$\,, $b^k_i=d^k_i=\zd^k_i$\,, $o(\nm{x}^2)=0$\,, $o(\nm{y}^2)=0$\,, $e^k_{ij}=f^k_{ij}=0$\,, so that
		$$(x\bullet y)^k=\zd^k_ix^i+\zd^k_jy^j+c^k_{ij}x^iy^j+o(||\bra{x}{y}\ran||)\,.$$
		Consequently,
		$$\lvec{\pa_{x^i}}=\pa_{x^i}+c^k_{ji}x^j\pa_{x^k}+v_i$$
and
$$\rvec{\pa_{x^i}}=\pa_{x^i}+c^k_{ij}x^j\pa_{x^k}+w_i\,,$$
		where $v_i,w_i$ is at least quadratic in $x$. Hence,
		$$[\lvec{\pa_{x^i}},\lvec{\pa_{x^j}}](0)=(c^k_{ij}-c^k_{ji})\pa_{x^k}$$
		and
		$$[\rvec{\pa_{x^i}},\rvec{\pa_{x^j}}](0)=(c^k_{ji}-c^k_{ij})\pa_{x^k}\,.$$
		As we see, the structure constants of the skew brackets are skew-symmetrization of $c^k_{ij}$ with respect to lower indices. Moreover,
$$[\lvec{\pa_{x^i}},\rvec{\pa_{x^j}}](0)=(c^k_{ji}-c^k_{ji})\pa_{x^k}=0\,.$$
	\end{proof}
	\begin{example}
		Consider the smooth inverse loop defined in the example \ref{e11}. If we take $\{\pa_{x^i}\}_{i=1,...,n}$ as a basis for $\mathbb R^n$, then the left and right prolongations are
		\[
		\lvec {\pa_{x^i}}(x)=\pa_{x^i}-\frac{1}{2}C^k_{ij}x^j\pa_{x^k\,},\quad \rvec {\pa_{x^i}}(x)=\pa_{x^i}+\frac{1}{2}C^k_{ij}x^j\pa_{x^k}\,,\quad x\in\mathbb R^n\,,
		\]
		where $x^j$'s are the coefficients of the vector field $x=x^j\pa_{x^j}\in \sT \mathbb R^n$\,.
		Then, the skew bracket associated to the skew algebra is
		\[
		[\lvec {\pa_{x^i}},\lvec {\pa_{x^j}}]=-[\rvec {\pa_{x^i}},\rvec {\pa_{x^j}}]=\frac{1}{2}(C_{ij}^k-C^k_{ji})\pa_{x^k}\,,
		\]
		where $\frac{1}{2}(C^k_{ij}-C^k_{ji})$ are structure constants of the skew symmetric bracket.
	\end{example}	
	We could not find an example of a loopoid with entirely different left and right tangent algebras, so we propose the following conjecture.
	 \begin{conjecture}\label{c1}
		The left and right skew algebroids of a smooth loopoid differ by sign, $[X,Y]_l=-[X,Y]_r$\,.
	\end{conjecture}
\begin{example}
		Consider the left inverse quasiloopoid of the Example \ref{almost}. We will characterize its infinitesimal counterpart. The source and the target maps are
		\[
		\za\left((a_1,b_1),(a_2,b_2)\right)=(a_1,b_1)\,, \quad \zb\left((a_1,b_1),(a_2,b_2)\right)=(a_2,b_2)\,.
		\]
		The tangent space of $G$ is
		\[
		\T G=\left\lbrace \left( \left( (a_1,b_1),(a_2,b_2)\right),\left( ( \dot a_1,\dot b_1),(\dot a_2,\dot b_2)\right)\right)  :\quad  \dot a_1-\dot a_2=\phi'(b_1-b_2)(\dot b_1-\dot b_2) \right\rbrace\,.
		\]
Since the projection $\alpha$ is constant on the first component, then $\dot a_1=\dot b_1=0$ which implies $\dot a_2=\phi '(b_1-b_2)\dot b_2$\,. On the other hand the target and the source maps coincide on the base $\mathbb R^2$ which means $a_1=a_2$ and $b_1=b_2$ on $\mathbb R^2$\,, so the vector bundle $AG=Ker(\T \za)\bigcap \T G$ over $\mathbb R^2$ is
		\[
		AG=\left\lbrace
		\left(
		\left( (a_1,b_1),(a_1,b_1)\right),\left( ( 0,0),(\phi'(0)\dot b_2,\dot b_2)\right)\right)  :\quad c(x,y):=\dot b_2
		\right\rbrace.
		\]
		Let $\{\partial_x,\partial_y\}$ be a basis of the tangent bundle of $\mathbb R^2$. The sections of $AG$ are of the form
		\[
		X=c(x,y)(\phi'(0)\partial_x+\partial_y)\,,
		\]
		which are one dimensional vector fields on $\mathbb R^2$. The vector field $Y=\phi'(0)\partial_x+\partial_y$ is the generator of bundle $AG$ and any other section of $AG$ is of the form $c(x,y)Y$ with the bracket
		\[
		[c_1(x,y)Y,c_2(x,y)Y]=c_1(x,y)\left( \phi'(0)\partial_x (c_2)+\partial_y (c_2)\right)Y+c_2(x,y)\left( \phi'(0)\partial_x (c_1)+\partial_y (c_1)\right)Y\,.
		\]
Thus, the skew algebroid structure on $AG$\,, as a vector bundle of rank one, depends only on the anchor map. Therefore,
		\[
		[X,fX]=(\rho(X)f)X=(\phi'(0)\partial_x f+\partial_y f)X\,, \quad f\in C^{\infty}(\mathbb R^2)\,,
		\]
		which means, $\rho (X)=\phi'(0)\partial_x+\partial_y$\,. Note that
		\[
		\rho[X,fX]=[\rho(X),\rho(fX)]=(\rho(X)f)\rho(X)\,.
		\]
		In this one dimension case the anchor map is an algebra homomorphism, therefore $AG$ is an almost-Lie algebroid.
		
	\end{example}

\begin{example}\label{Ex.loopoid}
		On $\mathbb R^2$ consider a smooth local loop structure $H$ defined by the multiplication
		\[
		x.y=(x_1,x_2).(y_1,y_2)=(x_1+y_1+x_1y_2,x_2+y_2+x_2y_1)\,,
		\]
		and the identity $(0,0)$\,. Taking $X_1=\pa_{x_1}$ and $X_2=\pa_{x_2}$ as a basis for $\sT_0 H$\,, we get
		\beas
		&\lvec {X_1}(x_1,x_2)=\pa_{x_1}+x_2\pa_{x_2}\,, \quad 	\lvec {X_2}(x_1,x_2)=x_1 \pa_{x_1}+\pa_{x_2}\,,\\ 	
&\rvec {X_1}(x_1,x_2)=(1+x_2)\pa_{x_1}\,,\quad	\rvec {X_2}(x_1,x_2)=(1+x_1)\pa_{x_2}\,.
		\eeas
Hence,
$$[\lvec{ X_1}, \lvec {X_2}]{(x_1,x_2)}=-[\rvec {X_1}, \rvec{X_2}]{(x_1,x_2)}=(\pa_{x_1}-\pa_{x_2})\,,$$
so that the bracket for the corresponding skew algebra $\mathfrak h\cong \sT_0 H$ is
		$$[X_1, X_2]=[\lvec{ X_1}, \lvec {X_2}]{(0,0)}=\pa_{x_1}-\pa_{x_2} =X_1-X_2\,.$$
		
		Using  the loop structure of $H$ and the pair groupoid $\mathbb R^2 \times\mathbb R^2 \rightrightarrows \mathbb R^2$, we define a smooth loopoid $G:=H\times \mathbb R^2\times \mathbb R^2$ over $M=\{(0,t,t)| t\in \mathbb R^2\}$ with the source map $\za(x,t,s)=(0,t,t)$\,, the target map $\zb(x,t,s)=(0,s,s)$\,, and the multiplication
		\[
		(x,t,s)\bullet(x',s,r)=(x.x',t,r)\,,
		\]
		which means
		\[
		((x_1,x_2),t,s)\bullet((x'_1,x'_2),s,r)=
		\left(  (x_1+x'_1+x_1x'_2,x_2+x'_2+x_2x'_1),t,r\right).
		\]
Let us introduce coordinates $(x_1,x_2,x_3,x_4,x_5,x_6)$ on the smooth loopoid $G=H\ti\R^2\ti\R^2$. The vector bundle $AG$ is the trivial vector bundle of rank four over $M=\R^2$\,. As a basis of its sections let us take the vector fields
 $$X_1=[\pa_{x_1}]\,,\quad  X_2=[\pa_{x_2}]\,,\quad X_3=[\pa_{x_3}]=[\pa_{x_5}]\,,\quad X_4=[\pa_{x_4}]=[\pa_{x_6}]\,,$$
along the submanifold $M$. Here, $[\pa_{x_i}]$ denotes the class of the vector field $\pa_{x_i}$ in the normal bundle of $M$ in $G$. Note that the representative $\pa_{x_1}\,,\pa_{x_2}\,,\pa_{x_5}\,,\pa_{x_6}$ are tangent to $\za$-fibers, and $\pa_{x_1}\,,\pa_{x_2}\,,\pa_{x_3}\,,\pa_{x_4}$ are tangent to $\zb$-fibers.
The left prolongations are
		\beas
		&\lvec{X_1}=\pa_{x_1}+x_2\pa_{x_2}\,, \quad \lvec{X_2}=x_1\pa_{x_1}+\pa_{x_2}\,,\\
		&\lvec{X_3}=\pa_{x_5}\,,\quad \lvec{X_4}=\pa_{x_6}\,.
		\eeas
		The only non-zero bracket is
$$[\lvec{X_1},\lvec{X_2}]=\pa_{x_1}-\pa_{x_2}=\lvec{X_1}-\lvec{X_2}\,.$$
Eventually, we get the brackets on sections of $A^lG$\,:
	\[
		{[X_1,X_2]}_l(0,t,t)=(\lvec {X_1}-\lvec {X_2})(0,t,t)=X_1(0)-X_2(0)\,.\]
 The rest of the brackets are zero.

\noindent Now, identifying the vector bundle $AG$ with $\sT\cF^\zb\,\big|_M$\,, we get the right prolongations
		\beas
		&\rvec {X_1}=(1+x_2)\pa_{x_1}\,, \quad \rvec {X_2}=(1+x_1)\pa_{x_2}\,,\\
		&\rvec {X_3}=\pa_{x_3}\,, \quad \rvec {X_4}=\pa_{x_4}\,.
		\eeas
The only non-zero bracket is
 $$[\rvec{X_1},\rvec{X_2}]=(1+x_2)\pa_{x_2}-(1+x_1)\pa_{x_1}=\rvec{X_2}-\rvec{X_1}\,.$$
This induces the following bracket on sections of $A^rG$:		
 		\[{[X_1,X_2]}_r{(0,t,t)}=\pa_{x_2}-\pa_{x_1}=(\rvec {X_2}-\rvec {X_1})(0,t,t)=X_2(0)-X_1(0)\,.\]
		Therefore, we have
		\[
		[X_1,X_2]_l=-[X_1,X_2]_r=X_1-X_2\,,
		\]
and the rest of the brackets equal zero.
		
	\end{example}
Note  that for Lie groupoids, the left and right algebroid structures differ just by sign. One of our main results is that this is also the case of smooth I.P. loopoids.  In other words the conjecture \ref{c1} is true for I.P. loopoids. Adapting (2) in \cite{MMS}, proved for Lie groupoids, we obtain the following.
\begin{lemma}\label{IP}
For any smooth I.P. loopoid $G$ with an inverse  $\zi:G\to G$\,, for any $X\in A_aG$ we have
\be\label{zi}\T_a\zi(\bar{X}^\za_a)=-\bar{X}^\zb_a\,.\ee
\end{lemma}	
\begin{proof}
Let $X$ be an element of $A_aG$\,. We take a curve $\zg:\R\to\cF^\za(a)$ such that its tangent vector $v$ at $a$ represents $X$\,, i.e, $\bar{X}^\za_a=v=\dot\zg(0)$\,.
 Let us observe that
$$\sT_{(a,a)}m(0,v)=\frac{\xd}{\xd t}m(a,\zg(t)|_{t=0}=v\,;$$
likewise for the $\zb$-path $\zg^{-1}(t)$ we have
$$\sT_{(a,a)}m(\sT_a\zi(v),0)=\sT_a\zi(v)\,.$$
Therefore
$$\sT_{(a,a)}m(\sT_a\zi(v),v)=v+\sT_a\zi(v)=\sT_a\zb(v)\,,$$
since
$m(\zg^{-1}(t),\zg(t))=\zb(\zg(t))$\,.
Hence,
$$\sT_a\zi(v)=-v+\sT_a\zb(v)\,.$$
 This proves (\ref{zi}), as $\sT_a\zb(v)$ is the anchor of $v$\,.
\end{proof}
\begin{remark}
Note that the above Lemma implies that $\sT\zi$ induces a map on the normal bundle
$$\widetilde{\sT\zi}:AG\to AG\,,\quad \widetilde{\sT\zi}(X)=-X\,.$$
\end{remark}
	 \begin{theorem}
		Let $G$ be a smooth I.P. loopoid. Then, $[X,Y]_l=-[X,Y]_r$\,.
	\end{theorem}
	\begin{proof}
		According to Lemma \ref{IP}, $\T_a\zi(\bar{X}^\za_a)=-\bar{X}^\zb_a$\,.
		As $\zi$ is the inverse, we have
		$$
		l_g=\zi\circ r_{g^{-1}}\circ \zi\,.
		$$
		Hence,
		$$\lvec{X}(g)=\T l_g(\bar{X}^\za)=\T\zi\circ\T r_{g^{-1}}\circ \T\zi(\bar{X}^\za)=\T\zi\circ\T r_{g^{-1}}(-\bar{X}^\zb)=-\T\zi(\rvec{X}(g^{-1}))=-\zi_*\rvec{X}(g)\,.
		$$
		In consequence, for $a\in M$,
		$${[X,Y]}_l(a)=[\lvec{X},\lvec{Y}](a)=[\zi_*(\rvec{X}),\zi_*(\rvec{Y})](a)=
\zi_*[\rvec{X},\rvec{Y}](a)=\sT_a\zi{[X,Y]}_r(a)=-{[X,Y]}_r(a)\,.$$

	\end{proof}
\begin{remark}
It is well known that the tangent algebra of a general analytical loop has a structure of an \emph{Akivis algebra} \cite{HS}. Unfortunately, the same Akivis algebra might correspond to non-isomorphic local analytic loops, so a finer structure is required in order to determine the analytical loop. The situation is similar in our case: a skew algebroid can be the infinitesimal part of many smooth loopoids. Whether additionally we have a ternary operation that makes our skew algebroid into an `Akivis algebroid' is not clear yet, but it is definitely an important question which is open up to now.
\end{remark}

\subsection{Functoriality}
\begin{lemma} If $B\rightrightarrows N$ is a smooth quasisubloopoid of the smooth quasiloopoid $G\rightrightarrows M$, then the natural map of normal
		bundles, induced functorially by the map of pairs $(B,N)\hookrightarrow(G,M)$, defines a morphism
		of skew algebroids $A^lB\hookrightarrow A^lG$ with left brackets. The same is true for the right brackets.
	\end{lemma}
	\begin{proof}  It is immediate that $AB\to N$ is a subbundle of $AG\to M$\,. Furthermore,
		$X\in\Sec(AG)$ restricts to a section of $AB$ if and only if the vector field $\lvec{X}$ is tangent to $B$\,. If
		$X,Y$ are two such sections, then $[\lvec{X},\lvec{Y}]$ again is tangent to $B$\,. That is, the space
		$\Sec(AG, AB)$ of sections $AG$ that restrict to sections of $AB$\,, is a skew subalgebra with respect to the left bracket. Similarly, for the right bracket.
		
	\end{proof}
	More generally, we have:
	\begin{theorem} If  $B\rightrightarrows N$ and $G\rightrightarrows M$  are smooth quasiloopoids, then the map on normal
		bundles, defined by quasiloopoid morphism of pairs $F : (B,N) \to (G,M)$\,, is a morphism of left and right skew algebroids
		$AF : AB\to AG$\,.
	\end{theorem}
	\begin{proof} $F$ is a morphism of quasiloopoids if and only if $\Graph(F) \subset B\ti G$ is a quasi subgroupoid.
		But then $$A^l(\Graph(F)) \subset A^l(B\ti G) = A^lB \ti A^lG$$ is a skew subalgebroid.
	\end{proof}
	In summary, we have constructed a \emph{generalized Lie functor} from the category of smooth quasiloopoids (and
	their morphisms) to pairs of skew algebroids (and their morphisms) with opposite anchors. Note however that in the nonassociative case our Lie functor is not an equivalence of categories. In general, a morphism of skew algebroids do not induce a morphism of the quasiloopoids; one skew algebroid can be the infinitesimal part of many (local) smooth quasiloopoids.
	
	\subsection{The case of loopoids}
	If we start from a loopoid $G\rightrightarrows M$\,, then
	$$
	\za(gh)=\za(g)\ \text{and}\ \zb(gh)=\zb(h)\,.
	$$
	This means that the \emph{anchor map}
	$$(\za,\zb):G\to M\ti M$$
	is a morphism of the loopoid $G$ into the Lie groupoid of pairs $M\ti M$\,. As the Lie algebroid of $M\ti M$ is $\T M$\,, we have natural morphisms of skew algebroids $$A^l(\za,\zb): A^lG\to\T M=A^l(M\ti M)\,.$$
	\begin{proposition}
		$A^l(\za,\zb)$ is the anchor of skew algebroid $A^lG\to M$.
	\end{proposition}
	\begin{proof} Let $X\in A^lG$ and $\bar{X}^\za\in\T_{\zb(g)}\cF^\za(\zb(g)$ be the corresponding vertical representant
		tangent to the $\za$-fiber. Let $\bar{X}^\za$ be represented by a smooth curve $\zg:\R\to G$, $\za(\zg(t))=\zb(g)$\,.
		The curve is mapped by the anchor map $(\za,\zb)$ to the curve $(\zb(g),\zb(\zg(t)))$. It represents the tangent vector $A^l(\za,\zb)(X)$ in $\T M$ viewed as $A^l(M\ti M)$\,.
		
	\end{proof}
	\begin{theorem}
	The left and right skew algebroid structures for a smooth loopoid $G\rightrightarrows M$ are almost Lie algebroids with opposite anchors.
	\end{theorem}
	\begin{proof} The anchor map $\zr=A^l(\za,\zb)$ is a morphism of skew algebroids covering the identity on $M$, therefore
		\be\label{almost1}\zr[X,X']_l=[\zr(X),\zr(X')]_l\,.\ee
	\end{proof}
	\begin{example}
Let $G=G'\ti N\ti N\rightrightarrows M$ be the smooth loopoid in the Example \ref{exa}. The submanifold of units is diffeomorphic with $N$ and the normal bundle $AG=\zn(G,M)$ can be identified with $AG'\ti \sT N$. If $X=(Y_0,Y)\in \Sec(AG)$, where $Y_0\in AG'$ and $Y$ is a vector field on $N$, then
$$\lvec{X}(g,x,y)=(\lvec{Y_0},0,Y)\quad\text{and}\quad \rvec{X}(g,x,y)=(\rvec{Y_0},-Y,0)\,,$$
so that $[X,X']_l=([Y_0,Y'_0]_l,[Y,Y'])$
The anchor map $\zr:AG\to \sT M$ reads $\zr(Y_0,Y)=Y$ and (\ref{almost1}) is satisfied.
\end{example}	
	
		\begin{example}
	Let $G\rightrightarrows M$ be a quasiloopoid (or loopoid) and $AG$ the corresponding skew algebroid. The prolongation of $AG$ over a fibration $\pi:P\to M$ which is defined by
\[
		\mathcal P^{\pi}(AG)=\{(X_m,Y_p)\in A_mG\times \sT_pP|\quad \zr(X_m)=\sT_p\pi(Y_p)\} \quad p\in P\,, \quad m=\pi(p)\,,
		\]
		is an almost Lie algebroid on $P$\,.
		
		Now, we denote by $A (\mathcal P^{\pi} G)$ the almost Lie algebroid of the loopoid $ \mathcal P^{\pi} G\rightrightarrows P$  defined in Remark \ref{prolong}. Then for $p\in P$ and $m=\pi(p)$\,, it follows that
 \[
A_p(\mathcal P^{\pi} G)=\{(0_p,X_m,Y_p)\in \sT_pP\times A_mG\times \sT_pP|\quad (\sT_p\pi)(Y_p)=(\sT_m\zb)(X_m)\}\,.
\]
One can see that there is a linear isomorphism
 $$A_p(\mathcal P^{\pi}G)\to \mathcal P_p^{\pi}(AG),\quad  (0_p,X_m,Y_p)\to (X_m,Y_p)$$
which induces an isomorphism between the skew algebroids $A(\mathcal P^{\pi}G)$ and $\mathcal P^{\pi}(AG))$\,. For more details see \cite{HiMa}.
		\end{example}

\section{Tangent and cotangent bundles of smooth loopoids}
	\subsection{ Tangent bundle of a quasiloopoid}
Like in the case of a Lie groupoid, the tangent bundle of a smooth quasiloopoid is a smooth quasiloopoid.
 The following is obvious.
\begin{lemma}\label{6.1}
If $\T\za(X_g)=\T\zb(Y_h)$\,, then $X_g$ and $Y_h$ can be represented by smooth curves $\zg_X, \zg_Y$ in $G$\,, $\zg_X(0)=g$\,, $\zg_Y(0)=h$\,, such that
$\za(\zg_X(t))=\zb(\zg_Y(t))$\,. In this case $\sT m(X_g,Y_h)$ is represented by $m(\zg_X(t),\zg_Y(t))=\zg_X(t)\cdot\zg_Y(t)$\,.
\end{lemma}
\noindent In the following theorem we adapt the proof of Theorem A.1 in \cite{MMS}.
	\begin{theorem}
		Let $G\rightrightarrows M$ be a quasiloopoid with the multiplication $m:G^{(2)}\to G$\,. For the composable pairs $(g,h)\in G^{(2)}$ denote $\zb(g)=\za(h)=q$\,. If $\tau$ and $\sigma$ are local $\za$-section and $\zb$-section defined in a neighborhood of the unit element $q\in M$ such that $g=\sigma(q)$ and $h=\tau(q)$\,, then the tangent multiplication $\T _{(g,h)}m:\T _{(g,h)}G^{(2)} \to \T _{(g,h)} G$ is given by
		\[
		\T _{(g,h)}m(v_{g},v_{h})=\T_g r_{\tau}(v_g)+\T_h l_{\sigma}(v_h)-\T_q (l_{\sigma}\circ r_{\tau})(v_q)\;,
		\]
		where $v_q=\T _g\zb(v_g)=\T _h \za (v_h)\,$. Obviously, for loopoids over a point (loops) the multiplication will be
		\begin{equation}\label{mT}
		\T _{(g,h)}m(v_{g},v_{h})=\T_g r_{h}(v_g)+\T_h l_{g}(v_h)\,.
		\end{equation}
	\end{theorem}
	\begin{proof}
		Consider the tangent vector $(v_g,v_h)\in \T _{(g,h)}G^{(2)}$ as the following decomposition
		\[
		(v_g,v_h)=(v_g,\T_q\tau(v_q))+(\T_q\sigma(v_q),v_h)-(\T_q\sigma(v_q),\T_q \tau(v_q))\,,
		\]
		and then apply the tangent map $\sT m_{(g,h)}$ to each term. For the first two terms the situation is rather obvious:
$$\sT m_{(g,h)}(v_g,\T_q\tau(v_q))=\T_g r_{\tau}(v_g)\quad\text{and}\quad \sT m_{(g,h)}(\T_q\sigma(v_q),v_h)=\T_h l_{\sigma}(v_h)\,.$$
 The third term requires more attention.
We will show that
		\begin{equation}\label{3}
		\sT_{(g,h)} m(\T_q\sigma(v_q),\T_q \tau(v_q))=\T_q(l_\sigma\circ r_\tau)(v_q)\,.
		\end{equation}
		If we take a curve $q(t)$ such that $v_q$ is tangent to it, then $\sigma$ and $\tau$ maps this curve to curves $g(t)$ and $h(t)$ for which $v_g$ and $v_h$ are tangent to them, respectively. So
	 \[
		\sT_{(g,h)} m(\T_q\sigma(v_q),\T_q \tau(v_q))=\frac{\xd}{\xd t}\big|_{t=0}m(g(t),h(t))=\frac{\xd}{\xd t}\big|_{t=0} m(\sigma(q(t)),\tau(q(t)))\,.
		\]
		But
		\[
		\begin{array}{rcl}
		(l_\sigma\circ r_\tau)(q(t))=l_\sigma(r_\tau(q(t)))&=&l_\sigma(q(t))\tau(\zb(q(t)))
		=l_\sigma(\tau(q(t)))\\[4pt]
		&=&\sigma(\za(\tau(q(t))))\tau(q(t))=m(\sigma(q(t)),\tau(q(t)))\,,
		\end{array}
		\]
		which proves (\ref{3}).
		
	\end{proof} 	
	Therefore, the lack of associativity is not an obstacle neither for definition of local bisections nor for the tangent multiplication in tangent bundle of quasiloopoids.
\begin{theorem} The tangent bundle $\sT G$ of a loopoid $G \rightrightarrows M$ is again a loopoid, with $\T M$ as the set of units and $\T\alpha,\T\beta : \T G\to \T M$ as the target and source maps. The multiplication map is $$\T m:\T G^{(2)}=(\T G)^{(2)}\ni(X_g, Y_h)\mapsto X_g\bullet Y_h \in \T_{gh} G\,.$$ If $G$ is an I.P. loopoid, then $\sT G$ is an I.P. loopoid.
  \end{theorem}
  \begin{proof}
Indeed, $\sT\za$ and $\sT\zb$ are surjective submersions. Moreover $\sT M\subset\sT G$ is the submanifold of units for the multiplication $\sT m\,$. Indeed,
if $X_a\in\sT_aM$ and $Y_g\in \cF^\za(\za(g))$ satisfying $\sT\za(Y_g)=X_a$, are represented by curves $a(t)\in M$ and $g(t)\in\cF^\za(\za(g)$ such that $a(t)=\za(g(t))$, then
 $$X_a\bullet Y_g=\sT m(X_a,Y_g)=\frac{\xd}{\xd t}\big|_{t=0}(a(t)\cdot g(t))=\frac{\xd}{\xd t}\big|_{t=0}(g(t))=Y_g\,.$$
Representing $X_g$ and $Y_h$, like in Lemma \ref{6.1}, by curves $\zg_X(t)$ and $\zg_Y(t)$\,, we have
$$\za(\zg_X(t)\cdot\zg_Y(t))=\za(\zg_X(t))\,.$$ After differentiating we get that
$$\sT\za(X_g\bullet Y_h)=\sT\za(X_g)\,.$$
Similarly
$$\sT\zb(X_g\bullet Y_h)=\sT\zb(Y_h)\,.$$
It remains to show that
\be\label{e0}l_{(X_g)}:\cF^{\sT\za}(\sT\zb(X_g))\to\cF^{\sT\za}(\sT\za(X_g))\ee
is a diffeomorphism (and similarly for $r_{(Y_h)}$)\,.

 First we show that $l_{(X_g)}$ is regular (local diffeomorphism), i.e, the differential
$$\sT_{(Y_h)}l_{(X_g)}:\sT_{(Y_h)}\left(\cF^{\sT\za}(\sT\zb(X_g))\right)\to\sT_{(X_g\bullet Y_h)}\left(\cF^{\sT\za}(\sT\za(X_g))\right)$$
has trivial kernel. Take a $\dot Y\in \sT_{(Y_h)}\left(\cF^{\sT\za}(\sT\zb(X_g))\right)$ and suppose
\be\label{e10}
\sT_{(Y_h)}l_{(X_g)}(\dot Y)=0_{(X_g\bullet Y_h)}\,.
\ee
We will show that $\dot Y=0_{(Y_h)}$\,. We can represent $\dot Y$ with a curve $Y_{h(t)}(t)\in \cF^{\sT\za}(\sT\zb(X_g))$\,, $h(0)=h$\,, $Y_h(0)=Y_h$\,.
Under the canonical projection $\sT\sT G\to\sT G$ the vector $\sT_{(Y_h)}l_{(X_g)}(\dot Y)$ goes to vector $\sT_h(l_g)(\dot h)$ represented by the curve $gh(t)$\,. In view of (\ref{e10}), $\sT_h(l_g)(\dot h)=0_{gh}$\,, so $\dot h=0_h$ because $l_g$ is a diffeomorphism. This means that we can reduce the curve $Y_{h(t)}(t)$ to the vertical curve $Y_h(t)\in\sT_h(\cF^\za(\zb(g))$\,, $Y_h(t)=Y_h+t\bar Y_h$\,, representing $\dot Y$. Of course $\sT\za(\bar Y_h)=0_{\za(h)}$\,. We have then that the curve
$$X_g\bullet Y_h(t)=\sT m(X_g,Y_h+t\bar Y_h)=\sT m(X_g+0_g,Y_h+t\bar Y_h)=X_g\bullet Y_h+0_g\bullet(t\bar Y_h)\,,$$
so $0_g\bullet(t\bar Y_h)$\,,
represents the vector $0_{gh}$\,. $\sT\za(\bar Y_h)=0$\,. Note that the vector $0_g\bullet\bar Y_h$ is represented by $g\zg(t)$, where $\zg(t)$ represents $\bar Y_h$\,, so $0_g\bullet\bar Y_h=\sT_h(l_g)(\bar Y_h)$ and
$$0_g\bullet(t\bar Y_h)=t\sT_h(l_g)(\bar Y_h)$$
represents the 0-vector, so $\bar Y_h=0$\,, as $l_g$ is a diffeomorphism. Hence, $\dot Y=0_{(Y_h)}$ and the map (\ref{e0}) is a local diffeomorphism.

 It remains to check global properties of $l_{(X_g)}$\,. First, we will show that it is injective. Indeed, suppose $X_g\bullet Y_h=X_g\bullet Y'_{h'}$\,. It is easy to see that in this case $h=h'$\,, so $Y'_h=Y_h+\bar Y_h$ and $\sT\za(\bar Y_h)=0_{\za(h)}$\,. As before, this means that $0_g\bullet \bar Y_h=\sT_h(l_g)(\bar Y_h)=0_{gh}$\,, so $\bar Y_h=0_h$\,.
Surjectivity of (\ref{e0}) is also clear. If we fix $Y_h$, then  $$X_g\bullet(Y_h+\bar Y_h)=X_g\bullet Y_h+0_g\bullet \bar Y_h\,,$$ where $\sT\za(\bar Y_h)=0_{\za(h)}$\,, and $0_g\bullet \bar Y_h=\sT_h(l_g)(\bar Y_h)$ can be arbitrary.

Finally, it is easy to see that if $\zi:G\to G$ is an inverse for $(G\rightrightarrows M,m)$, then $\sT\zi:\sT G\to \sT G$ is the inverse for $(\sT G\rightrightarrows\sT M,\sT m)$\,.

\end{proof}
\begin{remark} In the case of smooth loops an analogous result can be found in \cite{FN}. The structure on $\sT G$ is essentially the same as the one we propose here, however, the authors describe there the loop structure on $\sT G$ by means of the canonical trivialization $\sT G=G\ti\sT_eG$ rather than by means of the general differential geometric tangent functor.
\end{remark}

\subsection{ Cotangent bundle of a quasiloopoid}
 	The structure of $\T ^* G$ is much more complicated. It is well known that for any Lie group $G$, the cotangent bundle $\T ^*G$ has a natural structure of a Lie groupoid having the dual Lie algebra $\mathfrak g^*$ of $G$ as the unit submanifold. To go deeper in our understanding of structures of cotangent bundle of quasiloopoids, let us consider a quasiloopoid over a point which is actually a loop. Let $<G,\cdot , e>$ be a loop with $\mathfrak g$ as its corresponding skew algebra and $\mathfrak g^*$ as its dual skew algebra. In this case we can take the target and source maps $\za:=r^*_h$ and $\zb:=l^*_g$ to be pullbacks of smooth left and right translations, defined by
	\[
	(l^*_g\zx)(X)=\zx(\T_el_gX)\,,\quad (r^*_h\zx)(X)=\zx(\T_er_gX)\,,\quad X\in \mathfrak g, \zx\in \T _g^* G\,.
	\]
	
	Let $\mu_g\in \T^*_gG$ and $\nu_h\in \T^*_hG$ be composable pairs, that is $ \beta(\mu_g)= \alpha (\nu_h)=\sigma\in \mathfrak g^*$ or
	\begin{equation}\label{lr2}
	l^*_g(\mu_g)= r^*_h (\nu_h)=\sigma\,.
	\end{equation}
	The possible multiplication $\zm_g\bullet\zn_h\in\T^*_{gh}G$
	is defined by the equation
	\be\label{e4} \langle\zm_g,X_g\rangle+\langle\zn_h,Y_h\rangle=\langle\zm_g\bullet\zn_h, X_g\bullet Y_h\rangle\,,\quad  \mu_g\in \T ^*_gG\,,\quad\nu_h\in \T ^*_hG\,,
	\ee
	where $ X_g\bullet Y_h$ is the multiplication in the tangent loop (\ref{mT}).
	Denote $\zvy_{gh}=\zm_g\bullet\zn_h$\,, then
	$$
	\langle\zm_g,X_g\rangle+\langle\zn_h,Y_h\rangle=\langle\zvy_{gh},\T_gr_hX_g+\T_hl_gY_h\rangle=\langle r^*_h\zvy_{gh},X_g\rangle+\langle l^*_g\zvy_{gh},Y_h\rangle\,,
	$$
	implies $\zm_g=r^*_h\zvy_{gh}$ and $\zn_h= l^*_g\zvy_{gh}$ and then from (\ref{lr2}), we get
	\[
	l^*_g(r^*_h\zvy_{gh})= r^*_h (l^*_g\zvy_{gh})=\sigma\,,
	\]
	This is a sort of associativity which we do not have in our case.
	
	One may see that this problem comes from the fact that the equation (\ref{e4}) defines properly $\zm_g\bullet\zn_h$ if and only if the left-hand side vanishes for $(X_g,Y_h)$ in the kernel of multiplication map $\T m:\sT G^{(2)}\to G$ defined by (\ref{mT}). But $(X_g,Y_h)\in \ker(\sT m)$ is equivalent to
$$X_g=-(\T_g r_h)^{-1}(\T_h l_gY_h)\,.$$ If we insert the latter into (\ref{e4}), we get $l^*_g(r^*_h)^{-1}\zm_g=\zn_h$, which in the lack of associativity is a contradiction with the composability condition (\ref{lr2}). Eventually, we conclude that the standard way of defining the multiplication in $\T^*G$ does not give a well-defined product and that is why there is not a canonical quasiloopoid(loopoid) structure on the cotangent bundle $\T ^*G$ of a smooth quasiloopoid(loopoid) $G$.

Although the cotangent bundles of quasiloopoid(loopoid) have no canonical structures of quasiloopoid(loopoid), we still have good candidates for the source and the target maps $\tilde\alpha$ and $\tilde\beta$,
	$$\tilde\alpha,\tilde\beta:\T ^*G\to A^*G\,,$$ where  $A^\ast G$ is the dual bundle of the normal bundle $AG=\nu (G,M)$, defined by
	 \begin{equation}\label{eq:cotangent:groupoid}
	\kern-15pt
	\begin{array}{l}
	\langle\tilde\beta(\mu _g),X_{\beta(g)}\rangle=\langle\mu _g,\lvec{X}(g)
	\rangle\,, \quad\mu _g\in
	\T^\ast _gG \,,\quad X_{\beta(g)}\in A_{\beta(g)}G\,, \\[5pt]
	
	\langle\tilde\alpha (\nu
	_h),Y_{\alpha (h)}\rangle=\langle\nu _h,\rvec{Y}(h)\rangle\,,
	\quad\nu _h\in \T^\ast _hG\,,\quad Y_{\alpha (h)}\in A_{\alpha(h)}G\,.
	\end{array}
	\end{equation}
	 \begin{proposition}
For any smooth quasiloopoid $G$ the maps $\tilde\alpha,\tilde\beta:\T ^*G\to A^*G$ are smooth surjective submersions.
\end{proposition}
\begin{proof}
This is a direct consequence of Proposition \ref{preg}.
\end{proof}
	\begin{proposition}
		For any quasiloopoid $G\rightrightarrows M$, the cotangent bundle $\T^*G$ is fibred over the dual vector bundle $A^*G$ of associated skew algebroid $AG$ by the two fibrations $\tilde \za$ and $\tilde \zb$ defined by (\ref {eq:cotangent:groupoid}). Thus $\T^*G$ is a \emph{double fibred bundle}.
	\end{proposition}
	\begin{example}
		Consider a local smooth loop structure $G$ on $\mathbb R$ with the multiplication $x.y=x+y+x^2y$ and the target and source map $\alpha, \beta:G \to \{0\}\,$. The identity element is zero and there is no inverse.

\medskip		
		Let $(x,\dot x)$ be the canonical coordinates on $\T G\,$. We have the tangent maps of the left and right translations being $(\T_yl_x)(y,\dot y)=\dot y(1+x^2)$ and  $(\T_xr_y)(x,\dot x)=\dot x(1+2xy)\,$, respectively. The multiplication on the tangent bundle $\T G$ is therefore
		\[
		(x,\dot x)	\bullet (y,\dot y)=\left(x+y+x^2y\,, \dot x\, (1+2xy)+\dot y\,(1+x^2)\right)\,.
		\]
		The target and the source maps $\tilde \alpha , \tilde \beta: \T^*G \to (AG)^*=\mathbb R^*$ are defined by
		\[
		\begin{array}{rcl}
		\tilde \beta(x,p)(\dot y)&=&\left\langle p,\T_0l_x(\dot y)\right\rangle=p\,\dot y(1+x^2)\,,\\[4pt]
		\tilde \alpha(y,q)(\dot x)&=&\left\langle q,\T_0r_y(\dot x)\right\rangle=q\,\dot x\,.
		\end{array}
		\]
		Hence, $\tilde \beta(x,p)=p(1+x^2)$ and $\tilde \alpha(y,q)=q\,$.

	\end{example}

	\section{Discrete Lagrangian  mechanics on loopoids}
	In this section, we introduce Lagrangian mechanics on smooth loopoids. Let us first recall the discrete Lagrangian mechanics on Lie groupoids. Let $G\rightrightarrows M$ be a Lie groupoid, $\alpha, \beta : G \to M$ being its target and source maps, with a multiplication map $m : G^{(2)} \to G\,$, where $G^{(2)} = \{(g,h) \in G\times G| \; \beta(g) = \alpha(h)\}\,$. Denote its corresponding Lie algebroid by $AG$ represented by the normal bundle $\zn(M)=\sT G_{|M}/\sT M$ to the submanifold of units $M\subset G\,$. Sections of $AG$ are represented by the left-invariant $\lvec X$ (or right-invariant $\rvec X$) vector fields on $G$ associated with $X\in\Sec(AG)\,$. Discrete mechanics on a groupoid $G$ for a smooth, real-valued function $L$ on $G$ is defined as follows \cite{weinstein}.
	
	Let $L^{(2)}$ be the restriction to the set of composable pairs $G^{(2)}$ of the function $(g,h) \to L(g) + L(h)$ and $\Sigma_L \subset G^{(2)}$ be the set of critical points of $L^{(2)}$ along the fibers of the multiplication map $m$; that is, the points in $\Sigma_L$ are stationary points of the function $L(g)+L(h)$ when $g$ and $h$ are restricted to admissible pairs with the constraint that the product $gh$ is fixed. Variations of the constraint are of the form $(gu,u^{-1}h)\in G^{(2)}$.
	
	A solution of the Hamilton principle for the Lagrangian function $L$ is a sequence $...,g_{-2},g_{-1},g_0,g_1,g_2,...$ of elements of $G$\,, defined on some “interval” in $\mathbb Z$\,, such that $(g_i,g_{i+1}) \in \Sigma_L$ for each $i$\,.
	
	
	Discrete Euler-Lagrange equations on Lie groupoids can be derived from the variational principle. In \cite{MMM},  the discrete Euler-Lagrange equations take the form
	$$
	\lvec X(g_i)(L)-\rvec X(g_{i+1}) (L)=0
	$$
	on the Lie groupoid $G\rightrightarrows M$, for every section $X$ of $AG$. Note that $(g_i,g_{i+1})\in G^{(2)}$ and the left and right arrow denotes the right and left-invariant vector field on $G$ associated with $X\in \Sec(AG)$ understood as a section on the normal bundle.
	
	In our previous work \cite{GR} on the category of smooth loops, we indicate that because of the lack of associativity there is not clear variations like what we have in Lie groups. But still we can define the discrete Euler-Lagrange equations using the smooth prolongation of vector fields. It can be extended to the definition of Euler-Lagrange equations on loopoids.
	
	\begin{definition}
		The \emph{discrete Euler-Lagrange equation} for a discrete Lagrangian system on a smooth quasiloopoid $G\rightrightarrows M$ with Lagrangian $L:G\to \mathbb R$ is given by
		\begin{equation}\label{ELE.}
		\lvec X(L)(g_i)-\rvec X (L)(g_{i+1})=0,\quad( g_i,g_{i+1})\in G^{(2)}\,.
		\end{equation}
		for every $X\in A G$\,, where $\lvec X$ and $\rvec X$ are the left and right prolongation, respectively.
		
		A sequence $g_1,g_2,...$ of elements $G^{(2)}$ is a solution of the Euler-Lagrange equations if $g_i$ and $g_{i+1}$ satisfy (\ref{ELE.})  for $i=1,2,\dots$\, .
	\end{definition}
		
		The \emph{discrete Euler-Lagrange operator} $$DL(g,h): G^{(2)} \to A^*G\,,$$ where $A^*G$ is the dual of the skew algebroid $AG$,  is given by
		\[
		DL(g,h)(X)=	\lvec X(L)(g)-\rvec X (L)(h)\,.
		\]
		
A smooth mapping  $\gamma:G\to G$ is called a \emph{discrete flow} or \emph{discrete Lagrangian evolution operator} for a discrete Lagrangian $L$ if the graph

$$\graph(\gamma)=\{(g,\gamma(g))\,|\, g\in G\}$$
is a subset of $G^{(2)}$ and $(g,\gamma(g))$ is a solution of the discrete Euler-Lagrange equation for $L$ for all $g\in G$\,.
Equivalently, $	DL(g,\gamma(g))=0$ for all $g\in G$\,.
		
	We have the discrete Legendre transformation for a smooth loopoid $G$ which is similar to what we have for Lie groupoids \cite{MMS}.	
	Given a Lagrangian $L:G\to \mathbb R$ on a smooth loopoid $G\rightrightarrows M$ with the almost algebroid $AG$\,, we have two \emph{discrete Legendre transformations}

$$\mathbb F^+L=\tilde \zb\circ \xd L:G\to A^*G\quad\text{and}\quad\mathbb F^-L=\tilde \za\circ \xd L:G\to A^*G\,,$$ where $\xd L:G\to \sT ^*G$, and $\tilde \za$ and $\tilde \zb$ are the two fibrations of the double fiberd bundle $\sT^*G$ defined by
	\[
	\mathbb F^+L(g)(X)=\lvec X (L)(g)\,,\quad \mathbb F^-L(g)(X)=\rvec X (L)(g)\,,
	\]
 for $X\in\Sec(AG)$\,.
	Directly from the definitions, we get the following.
 \begin{proposition}\label{p2}
		A map $\gamma:G\to G$ is a discrete flow for the Lagrangian $L:G\to\R$ if and only if
		$$
		\mathbb F^-L\circ\zg=\mathbb F^+L \,.
		$$
	\end{proposition}

\begin{definition}
	 A discrete Lagrangian $L:G\to \mathbb R$ on a smooth loopoid $G\rightrightarrows M$ is called \emph{regular} if the Legendre transformation $\mathbb F^+L$ is a local diffeomorphism in a neighborhood of $M$\,. If $\mathbb F^+L$  is global diffeomorphism, $L$ is called to be \emph{hyperregular}.
	\end{definition}
\begin{theorem}\label{tl1} For an inverse smooth loopoid $G\rightrightarrows M$ the following are equivalent:
		\begin{itemize}
			\item A discrete Lagrangian $L:G\to \mathbb R$ on smooth loopoid $G\rightrightarrows M$ is regular;
			\item $\mathbb F^-L$ is a local diffeomorphism;
		\end{itemize}
		Moreover, $L$ is hyperregular if and only if $\mathbb F^- L$ is a global diffeomorphism. In this case the discrete Lagrangian evolution operator is a diffeomorphism.
	\end{theorem}
	The cotangent bundle $\sT ^*G$ is equipped with a canonical symplectic structure but as we observed the lack of associativity is an obstacle for defining a natural loopoid structure on $\sT^*G$ analogous to the Lie groupoid. Nevertheless, it still is a double fibred bundle. Setting aside the loopoid structure, for any function $L:G\to \mathbb R$ on manifold $G$ the submanifolds $\xd L(G)\subset \sT ^*G$ is a Lagrangian submanifold of the cotangent bundle. The discrete Euler-Lagrange dynamics can be equivalently described as follows.
	\begin{definition}
		Let $G\rightrightarrows M$ be a smooth loopoid and $L$ a discrete Lagrangian function on it. A sequence $ \mu_1,...,\mu_n\in \sT^*G$ satisfies the \emph{discrete Lagrangian dynamics} if $\mu_1,...,\mu_n$ are elements of the lagrangian submanifold  $\xd L(G)$ and they are composable sequence in $\sT ^*G$\,, that is
		$$
		\tilde\beta(\mu_k)= \tilde\alpha(\mu_{k+1})\,,\quad k=1,...,n-1\,.
		$$
	\end{definition}
	
	\begin{theorem}
		Let $G\rightrightarrows M$ be a smooth loopoid equipped with a discrete Lagrangian $L:G\to \mathbb R$\,. Then a sequence $ \mu_1,...,\mu_n\in \sT^*G$ satisfies the discrete Lagrangian dynamics if and only if
		\[
		\mu_k=\xd L(g_k)\quad \mbox{for some} \quad g_k\in G\,,\quad k=1,\dots,n\,,
		\]
 such that	$(g_k,g_{k+1})\in G^{(2)}$	and the discrete Euler-Lagrangian equations $\lvec X(g_k)(L)=\rvec X(g_{k+1})(L)$ are satisfied for $k=1,\dots,n-1$\,.
	\end{theorem}
	\begin{proof}
		It is enough to consider the discrete Legendre transforms of $L$ as $\mathbb F^+L= \tilde \beta \circ \xd L\,,$ $\mathbb F^-L= \tilde\alpha \circ \xd L:G\to \mathfrak g^*$. For more details consult \cite{MMS}.	
		
	\end{proof}
	\begin{example}
		Consider the smooth loopoid $G:=H\times \mathbb R^2\times \mathbb R^2$ over $M=\{(0,t,t)| t\in \mathbb R^2\}$ defined in Example \ref{Ex.loopoid}. Let $\{x_1,x_2,x_3,x_4,x_5,x_6\}$ be local coordinates for $G$\,. For the basis $X_i\in AG$ ($i=1,...,4$) of sections of $AG$ as in Example \ref{Ex.loopoid}, we have
		\beas
	&\lvec {X_1}(g)=\pa_{x_1}+x_2\pa_{x_2}\,, \quad \lvec {X_2}(g)=x_1\pa_{x_1}+\pa_{x_2}\,,\quad
	\lvec {X_3}(g)=\pa_{x_5}\,,\quad \lvec {X_4}(g)=\pa_{x_6}\,,\\
	& \rvec {X_1}(g)=(1+x_2)\pa_{x_1}\,, \quad \rvec {X_2}(g)=(1+x_1)\pa_{x_2}\,,\quad
	\rvec {X_3}(g)=\pa_{x_3}\,, \quad \rvec {X_4}(g)=\pa_{x_4}\,.
	\eeas
 If we take the Lagrangian function $$L=\sum_{i=1}^6\frac{ x_i^2}{2}$$ as the `total kinetic energy' of the system, then, for composable $g=(x_1,x_2,x_3,x_4,x_5,x_6)$ and $h=(x'_1,x'_2,x_5,x_6,x'_3,x'_4)$\,, the discrete Euler-Lagrange equations read
 \[
 x_1+x_2^2=(1+x'_2)x'_1\,, \quad x_1^2+x_2=(1+x'_1)x'_2\,,\quad x_5=x'_3\,,\quad x_6=x'_4\,.
 \]	
One can see that each sequence of composable pairs $g_i,g_{i+1},...$ of the form
\[
g_i=(x_1^i,x_2^i,a,b,c,d)\,,\quad g_{i+1}=(x_1^{i+1},x_2^{i+1},c,d,c,d)\,,\quad a,b,c,d\in\mathbb R\,,
\]
where
\[
 x^i_1+(x^i_2)^2=(1+x^{i+1}_2)x^{i+1}_1\,, \quad (x_1^i)^2+x^i_2=(1+x^{i+1}_1)x^{i+1}_2\,,
\]		
is a solution of the discrete Euler-Lagrange equations. For instance, if $x_1^i=1$ and $x_2^i=2$\,, then we get the following sequence of composable pairs
\beas &(1,2,a,b,c,d)\,,\left(\frac{1}{2}(1+\sqrt{21})\,,\frac{1}{2}(\sqrt{21}-3),c,d,c,d\right)\,,\\
&\left(\frac{3}{2}-\sqrt{21}+\frac{1}{2}\sqrt{125-16\sqrt{21}}\,,-\frac{5}{2}+
\sqrt{21}+\frac{1}{2}\sqrt{125-16\sqrt{21}},c,d,c,d\right)\,,\dots\,.
\eeas

To check whether  the Lagrangian $L$ is regular or hyperregular, we need to find the Legendre maps associated with $L$\,. The Legendre transformations are
\beas
	&\mathbb F^+L(g)=( x_1+x_2^2)X^1+(x_1^2+x_2)X^2+x_5 X^3+x_6X^4\,,\\
&\mathbb F^-L(g)=(x_1+x_2x_1)X^1+(x_2+x_1x_2)X^2+x_3 X^3+x_4X^4\,,
\eeas
where $(X^i)$ is a basis of sections of $A^*G$ dual to $(X_i)$\,. We have
\[
 \mathbb F^-L(\za(g))=\mathbb F^+L(\za(g))= x_3X^3+x_4X^4\,,\quad \za(g)=(0,0,x_3,x_4,x_3,x_4)\,.
\]
 Taking the derivative at the identity element $\za(g)$\,, we get
\[
\sT_{\za(g)}\mathbb F^+L(\pa_{x_1})=\frac{\xd}{\xd t}\big|_{t=0}\mathbb F^+L(t,0,x_3,x_4,x_3,x_4)=\frac{\xd}{\xd t}\big|_{t=0}(t\,X^1+t^2\,X_2+x_3\, X^3+x_4\,X^4)=X^1\,.
\]
Analogously,
\[
\sT_{\za(g)}\mathbb F^+L(\pa_{x_2})=\frac{\xd}{\xd t}\big|_{t=0}\mathbb F^+L(0,t,x_3,x_4,x_3,x_4)=\frac{\xd}{\xd t}\big|_{t=0}(t^2\,X^1+t\,X_2+x_3\, X^3+x_4\,X^4)=X^2\,.
\]
 By similar calculations we get
$$\sT_{\za(g)}\mathbb F^+L(\pa_{x_i})=\sT_{\zb(g)}\mathbb F^+L(\pa_{x_i})=0$$
for $i=3,4$ and
$$\sT_{\za(g)}\mathbb F^+L(\pa_{x_i})=\sT_{\zb(g)}\mathbb F^+L(\pa_{x_i})=X^{i-2}$$
for $i=5,6$\,.
Analogously,
\[
\sT_{\za(g)}\mathbb F^-L(\pa_{x_i})=X^i
\]
for $i=1,2,3,4$ and
 \[
\sT_{\za(g)}\mathbb F^-L(\pa_{x_i})=0
 \]
for $i=5,6$\,.

\medskip\noindent  This means that the Legendre transformations are local diffeomorphism at the neighborhood of identity elements along the fibers, which implies the Lagrangian $L$ is regular. It can be also computed by mathematical software such as Maple that as a discrete flow $\gamma:G\to G$ for the Lagrangian $L$ can be taken the map
			\[
\gamma(x,y,a,b,c,d)=(z,w,c,d,c,d)\,,		
			\]
such that
$$
z=-\frac{1}{2}\left(1+x+y+x^2+y^2+\sqrt{B}\right)\,,
$$
$$w=\frac{1+x+y+y^2+x^2-\sqrt{B}}{-1+x-y-x^2+y^2+\sqrt{B}}\,,
$$
	
where
\[
B=1+2(x+y-yx+x^2y+xy^2-x^2y^2-x^3-y^3)+3(x^2+y^2)+x^4+y^4.
\]

	\end{example}
	
	
\section{Conclusions}
We have introduced the categories of smooth quasiloopoids and loopoids as nonassociative analogs of Lie groupoids. We have shown that a big part of the Lie theory of Lie groupoids (e.g. Lie functor) can be generalized to this nonassociative case. In particular, infinitesimal parts of smooth loopoids have been recognized as skew algebroids.
We have not obtained ternary structures on these skew algebroids producing `Akivis algebroids'. This question we postpone to a forthcoming article.
We have studied the properties of the tangent and cotangent bundles of smooth loopoids showing that the tangent bundles are smooth loopoids themselves. Smooth quasiloopoids and loopoids provide also a geometric framework for discrete Lagrangian mechanics.

One can also use quasiloopoids and loopoids as a playground for information geometry (cf. \cite{Amari,GGKM}). The contrast (potential) function
is in this picture a function $F:G\ra \R$ on the quasiloopoid/loopoid $G\rightrightarrows M$ with vanishing first jets along $M$. This defines a quasi-Riemannian metric $g$ (an analog of the Fisher-Rao metric) on the corresponding skew algebroid $AG$\,, $$g(X,Y)=\overset{\leftarrow}{X}\overset{\leftarrow}{Y}(F)_{|M}\,.$$
Since linear $AG$-connections and the Levi-Civita connection are well-defined even for skew-algebroids, we can then try to find the dualistic pair of such connections, the difference
of which will define a symmetric covariant three-tensor on $AG$\,. This should lead to `information geometry on loopoids'.
	

	\small{\vskip1cm
		
		\noindent Janusz GRABOWSKI\\ Institute of
		Mathematics\\  Polish Academy of Sciences\\ \'Sniadeckich 8, 00-656 Warszawa, Poland
		\\Email: jagrab@impan.pl \\
		
		\noindent Zohreh RAVANPAK\\ Institute of
		Mathematics\\  Polish Academy of Sciences\\ \'Sniadeckich 8, 00-656 Warszawa, Poland
		\\Email: zravanpak@impan.pl \\

	\end{document}